\documentclass[11pt,letterpaper]{article}

\usepackage{amsmath, amsfonts,amsthm}

\usepackage[english]{babel}

\usepackage{amssymb,graphicx}
\usepackage{verbatim}
\usepackage{mathrsfs}
\usepackage{epsfig}

\usepackage{xcolor}
\usepackage{tikz}
\usepackage{pgfplots}
\usepgfplotslibrary{fillbetween}

\usepackage{url}
\usepackage{bm}
\usepackage[numbers]{natbib}

\usepackage{txfonts,enumitem, comment}

\usepackage[colorlinks,linkcolor=blue,citecolor=blue,anchorcolor=blue]{hyperref}

\newtheorem{theorem}{Theorem}[section]

\newtheorem{lemma}[theorem]{Lemma}
\newtheorem{proposition}[theorem]{Proposition}

\newtheorem{definition}[theorem]{Definition}
\newtheorem{remark}[theorem]{Remark}

\numberwithin{equation}{section}

\DeclareMathOperator*{\esssup}{ess\,sup}

\definecolor{Red}{rgb}{1.00, 0.00, 0.00}

\definecolor{DRed}{rgb}{0.5, 0.00, 0.00}

\definecolor{Blue}{rgb}{0.00, 0.00, 1.00}

\definecolor{Green}{rgb}{0.0, 0.4, 0.0}

\definecolor{Gray}{rgb}{0.7421875,0.7421875,0.7421875}

\definecolor{mycolor1}{rgb}{0.00000,0.44700,0.74100}%
\definecolor{mycolor2}{rgb}{0.85000,0.32500,0.09800}%
\definecolor{mycolor3}{rgb}{0.92900,0.69400,0.12500}%
\definecolor{mycolor4}{rgb}{0.49400,0.18400,0.55600}%
\definecolor{mycolor5}{rgb}{0.46600,0.67400,0.18800}%
\definecolor{mycolor6}{rgb}{0.00000,1.00000,1.00000}%

\allowdisplaybreaks[4]
\pgfplotsset{compat=newest}

\def\beq{\begin{eqnarray}}
\def\eeq{\end{eqnarray}}
\def\be*{\begin{eqnarray*}}
\def\ee*{\end{eqnarray*}}

\def \A{\mathbb{A}}

\def \E{\mathbb{E}}

\def \R{\mathbb{R}}

\newcommand{\Pbb}{\mathbb P}
\newcommand{\1}{\mathbf 1}

\newcommand{\cS}{\mathcal S}
\newcommand{\cT}{\mathcal T}

\newcommand{\dd}{\,\mathrm d}

\def\E{\mathbb{E}}

\def\R{\mathbb{R}}

\def\namedlabel#1#2{\begingroup
	#2%
	\def\@currentlabel{#2}%
	\phantomsection\label{#1}\endgroup
}

\title{\Large \bf
Distribution-constrained Maximum Stopping of Maximum Type
}
\date{\vspace{-5ex}}

\author{Shuoqing Deng
\thanks{
{\small The Hong Kong University of Science and Technology, Department of Mathematics (masdeng@ust.hk). S. Deng is supported by the Hong Kong University of Science and Technology Start-up grant no. R9826 and Hong Kong RGC Early Career Scheme (ECS) under grant no. 26307125. }}%
\and 
Xin Zhang
\thanks{
{\small 
New York University, Department of Finance and Risk Engineering (xz1662@nyu.edu). X. Zhang is supported in part by the National Science
Foundation under grant DMS-2508556. }}
}

\begin{document}


\maketitle

\begin{abstract}


We consider the distribution-constrained optimal stopping problem $\sup_{\tau\sim \mu} \mathbb E[B^*_\tau]$, where $\mu$ is a probability distribution on $\mathbb R_+$, and $(B^*_t)$ denotes the running maximum of a standard Brownian motion. This problem was introduced in \cite{BEES2018}, where a monotonicity principle is used to establish the optimal stopping time as the hitting time of a specific boundary.

In this paper, we characterize this boundary by a variational inequality. 
 In the spirit of \cite{COT}, we provide a \emph{novel} probabilistic representation for the variational inequality as a time-reversed optimal stopping problem. A key ingredient for proving the viscosity solution property and comparison principle is a quantitative estimate near the singular corner of the time-space domain, where the initial and boundary conditions are incompatible. We then prove the optimality of the resulting hitting time through a discrete-time Snell envelope construction and a stability argument for the associated stopping times.

\vspace{0.6 cm}


\vspace{0.2 cm}


\end{abstract}

\section{Introduction}

Let $B=(B_t)_{t\ge0}$ be a standard Brownian motion with $B_0=0$
on a filtered probability space satisfying the usual conditions, and denote its running maximum by 
 $$M_t:=B_t^*:=\sup_{0\le s\le t}B_s.$$  Let  $\mu$ be  an atomless probability measure on $\R_+$, we look for a stopping time $\tau$ for Brownian motion $B$ such that $\tau \sim \mu$ and $\tau$ maximizes $\E[B_\tau^*]$; i.e. we aim to solve
\begin{equation} \label{eq:max_max}
\sup_{\tau \sim \mu}\E[B_\tau^*].
\end{equation}

Such problem was introduced by Beiglb{\"o}ck et al.  \cite{BEES2018}. More generally, it belongs to a  class of the so-called \textit{distribution-constrained optimal stopping problem}, which  was first considered by Bayraktar and Miller \cite{BM2019} when $\mu$ is discrete. It aims at finding  a stopping time $\tau$ for Brownian motion $B$ such that $\tau \sim \mu$ and $\tau$ maximizes or minimizes the objective
$$
\E[U(B_{\cdot \wedge \tau}, \tau)],
$$
for some measurable  function $U$. In the following, we will refer to \eqref{eq:max_max} as the \textit{distribution-constrained maximum stopping of maximum type}, to distinguish with minimization problem $\inf_{\tau \sim \mu}\E[B_t^*]$ which is different in nature, see \cite[Remark 7.1]{BEES2018}.

When $\mu$ is discrete, Bayraktar and Miller \cite{BM2019} solved the distribution-constrained optimal stopping  using the dynamic programming principle (DPP) approach. K\"allblad \cite{K2022} then extended it to general distributions $\mu$ using measure-valued martingale (MVM). The MVM formulation allows the original terminal condition $\tau \sim \mu$ to become an initial condition and  restore the DPP.

Another approach to distribution-constrained optimal stopping was established around the same time by Beiglb\"ock et al. \cite{BEES2018}. They borrowed ideas from the optimal Skorokhod embedding problem (SEP) \cite{BCH2017} and provided a geometric characterization of the optimal stopping time $\tau$. Following their approach, \cite{DengHuang} studied the \textit{Root-type solution} to the distribution-constrained multiple stopping problem, established an time-reversal optimal-stopping representation in the spirit of \cite{COT} and more general optimality results (in terms of both the cost function and the process).

To state our main result, let us introduce
$$Y_t:=B_t-M_t\le0.$$ 
The process $Y$ is a Brownian motion reflected downwards at $0$ on $(-\infty,0]$.
Its generator in the interior is $\frac12\partial_{xx}$, with the
Neumann boundary condition at $x=0$.
Let $\mu$ be an atomless probability measure on $[0,\infty)$, and define 
 $p(t):=\mu((t,\infty))$ for $t\ge0$.
Function $p$ is continuous, nonincreasing with $p(0)=1$, and
$p(t)\to0$ as $t\to\infty$.
The atomlessness assumption ensures that the obstacle $p$ in the
variational inequality is continuous, and  $\Pbb(\tau\ge t)=\Pbb(\tau>t)=p(t)$  for any
stopping time $\tau$ with law $\mu$, which are important for the recovery of the law in Section \ref{sec:rocover_law}.
In \cite{BEES2018}, it was proved that the optimal stopping time is the hitting time of process $Y$, which can also be justified formally by a duality argument as in Remark~\ref{rmk:duality}. For a boundary $b:[0,\infty)\to[-\infty,0]$, let
\begin{equation}\label{eq:tau-b}
 \tau_b:=\inf\{t>0:Y_t\le b(t)\}
       =\inf\{t>0:B_t-B_t^*\le b(t)\},
 \qquad \inf\varnothing:=\infty.
\end{equation}
We are interested in: (i) construct a boundary $\hat b$ from a well-posed variational inequality; (ii) prove $\Pbb(\tau_{\hat b}>t)=p(t)$, hence $\tau_{\hat b}\sim\mu$; (iii) prove that $\tau_{\hat b}$ maximizes $\E[M_\tau]$ among all $\tau\sim\mu$. Actually,  $\tau_{\hat b}$ solves distribution-constrained optimal stopping problems with more general objective functions which are in the form $cM_t+h(B_t-M_t)$, where $c>0$ is a constant and $h$ satisfies suitable conditions; see Remark~\ref{rmk:duality}.

The distribution-constrained optimal stopping problem is closely related to Skorokhod embedding problem (SEP) and Martingale optimal transport (MOT). In particular, it shares the same feature as SEP in that the  solution depends on the cost function: the Root solution \cite{Root1969, COT}, the Rost solution \cite{Rost1976}, the Az{\'e}ma-Yor solution \cite{AY1979,BHR, OS2017}, and the Perkins/Hobson-Pedersen solution \cite{P1985, HP2002} solve different optimal SEP problems. The monotonicity principle of the distribution-constrained stopping in \cite{BEES2018} also share similarity with that of \cite{BCH2017}.  In \cite{DengHuang}, it was established that the \textit{Root-type solution} of distribution-constrained multiple stopping has a similar probabilistic representation as the Root solution of multi-marginal SEP \cite{COT}. 

The distribution-constrained optimal stopping problem can be related to many fundamental problems in mathematical finance, probability, and stochastic control. When first introduced in \cite{BM2019}, Bayraktar and Miller applied the problem to study the model-free super-hedging of a volatility outlook. In \cite{BEES2018}, it was established that the solution to the problem is closely associated to \textit{Shiryaev's problem}, or the inverse first-passage problem. Motivated by the work of Avellaneda and Zhu \cite{AZ2001} on credit-risk modeling, Chen et al.~\cite{CCCS1, CCCS} gave a rigorous characterization of the barrier using a variational inequality. 
In addition, the distribution-constrained stopping problem can also be viewed as an optimal stopping adaptation of the mean-field planning problem of P. L. Lions; see \cite{Lions, RTTY}.


Our main contribution in this work is the following.
First, we analyze the wellposedness of a variational inequality which determines the boundary of the  inverse first passage problem. For proving the existence and uniqueness of the viscosity solution,   inspired by \cite{COT, DengHuang} we introduce a novel probabilistic representation for the variational inequality as a time-reversed optimal stopping problem. In the variational inequality, the initial and boundary conditions are incompatible at the origin.  We address this issue by establishing quantitative bounds that determine an admissible solution class and allow comparison without assigning a value at the origin $(0,0)$.

Second, we give a direct proof of optimality for the constructed stopping time. We first work in discrete time and add a deterministic
function of the stopping time to the reward. The function is chosen in a way so that the corresponding barrier stopping time is optimal for the resulting unconstrained stopping problem. We verify optimality using the Snell envelope and then pass to continuous time through the stability of the associated
stopping times. This provides an alternative to the geometric approach of \cite{BEES2018} without requiring the moment assumption on $\mu$.

\noindent \textbf{Organization of the paper.} In Section \ref{sec:main}, we will first present the main results of this paper. In Section \ref{sec:proof1},  Section \ref{sec:proof2} and Section \ref{sec:proof3}, we will respectively prove the wellposedness of the variational inequality, the recovery of the correct law, and the optimality of the construction.

\section{Main results} \label{sec:main}

\subsection{The  free-boundary problem}

In this section, using similar idea as \cite{AZ2001}, we formally derive the variational inequality related to the boundary of the inverse first passage problem for $Y_t$.

\subsubsection{The PDE formulation}

In this formal derivation, we assume that $b$ is smooth and strictly negative and that
the law of $Y_t$ restricted to $\{\tau_b>t\}$ has a density $u$, namely $\mathbb P(\tau_b>t,Y_t\in\dd x)=u(t,x)\dd x$, for $b(t)<x<0$.
It follows that the density function $u$ satisfies the forward equation 
\begin{equation}\label{eq:density-system}
\begin{cases}
 u_t=\frac12u_{xx},&b(t)<x<0,\\
 u_x(t,0)=0,\quad u(t,b(t))=0,&t>0,\\
 u(0,\cdot)=\delta_0.
\end{cases}
\end{equation}
For $b(t)<x\le0$, define function 
 $F(t,x):=\mathbb P(\tau_b>t,Y_t>x)
        =\int_x^0u(t,y)\dd y$; 
and for $x\le b(t)$ let $F(t,x)=\mathbb P(\tau_b>t)$.
Then, for $b(t)<x<0$, we have
\begin{equation}\label{eq:F-heat}
 F_t(t,x)=\int_x^0u_t(t,y)\dd y
 =\frac12\bigl(u_x(t,0)-u_x(t,x)\bigr)
 =-\frac12u_x(t,x)=\frac12F_{xx}(t,x).
\end{equation}
The boundary conditions are
\[
 F(t,0)=0,\qquad F(t,b(t))=\mathbb P(\tau_b>t),\qquad
 F_x(t,b(t)+)=-u(t,b(t))=0.
\]
Writing the equation for $F$ converts reflection at $0$ into the Dirichlet condition $F(t,0)=0$. When $p(t):=\mu((t,\infty))$ is the target survival function, the
classical free-boundary problem becomes
\begin{equation}\label{eq:classical-FBP}
\begin{cases}
 F_t=\frac12F_{xx},&b(t)<x<0,\ t>0,\\
 F(t,x)=p(t),&x\le b(t),\ t>0,\\
 F(t,0)=0,\quad F_x(t,b(t)+)=0,&t>0,\\
F(0,x)=1,&x<0.
\end{cases}
\end{equation}
The initial and boundary conditions are incompatible at $(0,0)$ because $Y_0=0$.

\subsubsection{Variational inequality formulation}

Similar as \cite{CCCS}, we introduce the variational inequality form of \eqref{eq:classical-FBP}. Since $F=p$ in the stopping region and $p$ is nonincreasing, we have that  $F_t-\frac12F_{xx}=p'(t)\le0$.
 In the continuation region $F<p$ and
\eqref{eq:F-heat} holds.  This motivates the following variational inequality formulation
of the classical free-boundary problem:
\begin{equation}\label{eq:VI-F}
    \max\left\{
       F_t-\frac12F_{xx},\ F-p(t)
    \right\}=0,
    \qquad \mbox{ for all } x<0,\ t>0,
\end{equation}
with the initial and boundary conditions 
\begin{equation}\label{eq:VI-data}
    F(0,x)=1,  
    \qquad
    F(t,0)=0, \qquad \mbox{ for all } x<0,\ t>0.
\end{equation}
The boundary constructed from the variational-inequality solution is denoted by $\hat b$:
\begin{equation}\label{eq:b-stopping-boundary}
    \hat b(t):=\sup\{x\le 0:F(t,x)=p(t)\}\in[-\infty,0],
\end{equation}
and we use the convention $\sup\varnothing=-\infty$.
We distinguish this constructed boundary from a general boundary $b$.
The value $\hat b(t)=-\infty$ means that the stopping region at time $t$ is empty.
We write $\hat\tau:=\tau_{\hat b}$ for the associated first entry time.

\subsection{Wellposedness of the variational inequality}

In this section, we justify the existence and uniqueness of viscosity solution to the system \eqref{eq:VI-F}--\eqref{eq:VI-data}. Inspired by \cite{COT, DengHuang}, we first introduce a candidate probabilistic representation for the solution to the variational inequality, which has the interpretation as the value function of a \textit{time-reversed optimal stopping} problem.
Then we make some estimation of the value function at the singular point $(0,0)$. Finally we justify the backward stopping problem provides the uniques viscosity solution of the variational inequality.

\subsubsection{A time-reversed optimal stopping problem}

Let $W$ be a Brownian motion.  For $x\le 0$, let
\[
    X_s^x:=x+W_s,
    \qquad
    \sigma_0^x:=\inf\{s\ge0:X_s^x=0\}.
\]
The process is sent to a cemetery state after $\sigma_0^x$.  In other words, $X$ is
a Brownian motion absorbed at zero, instead of a reflected Brownian motion.  For $t\ge0$
and $x\le 0$, we introduce: 
\begin{equation}\label{eq:backward-value}
    V(t,x)
    :=
    \inf_{\theta\in\cT,\, 0\le\theta\le t}
    \E\left[
        p(t-\theta) \1_{\{\theta<\sigma_0^x\}}
    \right].
\end{equation}
Here $t$ is the amount of remaining backward time.  Immediate stopping $\theta=0$ costs $p(t)$, and hence $V(t,x) \leq p(t)$. It is clear that $V(0,x)=1$ for any $x<0$, and $V(t,0)=0$ for any $t>0$. Therefore $V$ has initial and boundary conditions \eqref{eq:VI-data}.
\begin{remark}
In the following, we will justify that $V$ is exactly the unique viscosity solution to \eqref{eq:VI-F}--\eqref{eq:VI-data}.
Indeed, it can be seen as an optimal stopping problem with respect to a backward process. The direction of time in \eqref{eq:backward-value} is in consistent with the sign in
\eqref{eq:VI-F}. This probabilistic representation to the inverse first passage problem follows the same spirit as \cite{COT, DengHuang}, see also \cite{CW}.

\end{remark}


\vspace{2mm}

\noindent\textbf{Regularity of $V$ at $(0,0)$}
 
We notice that the initial and boundary conditions are inconsistent at the point $(0,0)$: the initial condition requires $V(0,x)=1$ for $x<0$, whereas the boundary condition gives $V(t,0)=0$ for $t>0$. Consequently, one cannot expect continuity at this point. In the following lemma, we will obtain some estimations of $V$ near $(0,0)$, which will be used later in the comparison argument in the proof of viscosity solution.
For all $t>0,\ x\le 0$, let $H(t,x):=\Pbb(\sigma_0^x>t)
    =2\Phi\!\left(\frac{-x}{\sqrt t}\right)-1$,
where $\Phi$ is the standard normal distribution function.  The function $H$
solves
\[
    H_t=\frac12H_{xx},
    \qquad
    H(t,0)=0,
    \qquad
    \lim_{t\to 0}H(t,x)=1 \mbox{ for all } x<0.
\]
We have the following bounds for the value function $V$, which controls its behavior near  $(0,0)$.
\begin{lemma}\label{lem:corner-bounds}
The value function \eqref{eq:backward-value} satisfies
\begin{equation}\label{eq:corner-bounds}
    p(t)H(t,x)\le V(t,x)\le\min\{p(t),H(t,x)\}.
\end{equation}
Consequently,
\begin{equation}\label{eq:corner-control}
    |V(t,x)-H(t,x)|\le1-p(t)
\end{equation}
whenever $H(t,x)\le p(t)$.
\end{lemma}

Motivated by the above lemma, we will restrict in the next section the admissible sub- and supersolutions \(U\) to satisfy
$$
    p(t)H(t,x)\le U(t,x)\le \min\{p(t),H(t,x)\}.
$$
In the comparison argument below, when we consider sub- and supersolutions \(u\) and \(v\) satisfying satisfy the above bounds, it will naturally imply that
\begin{equation}\label{eq:corner-comparison-control}
u(t,x)-v(t,x)
\le (1-p(t))H(t,x)
\le 1-p(t)\to 0
\end{equation}
as \( (t,x) \to (0, 0)\). Consequently, there is no positive gap between \(u(t,x)\) and \(v(t,x)\) as  $(t,x) \to (0, 0)$.

\vspace{2mm}

\noindent\textbf{Dynamic programming principle}

We show the following stopped form of the dynamic programming principle,  which will be useful for the proof of viscosity solution property. We use the convention
\[
    X_s^x=0 \quad\text{for }s\geq\sigma_0^x,
    \qquad V(q,0)=0 \mbox{ for } q>0.
\]

\begin{proposition}\label{lem:local-continuation-DPP}
Let $(t,x)\in(0,\infty)\times(-\infty,0)$, and let $\rho$ be a stopping time such that
$0\leq\rho\leq t$.  Then
\begin{equation}\label{eq:stopped-waiting-inequality}
    V(t,x)
    \leq
    \E\!\left[V(t-\rho,X_\rho^x)
                  \1_{\{\rho<\sigma_0^x\}}\right].
\end{equation}
Moreover, let $Q$ be an open neighborhood of $(t,x)$ such that
\[
    \overline Q\subset(0,\infty)\times(-\infty,0),
    \qquad
    V(q,y)<p(q)\quad\text{for all }(q,y)\in\overline Q.
\]
Define $\nu_Q:=\inf\{s\geq0:(t-s,X_s^x)\notin Q\}$. 
For every sufficiently small $h>0$, let
$\zeta_h:=h\wedge\nu_Q$, one has
\begin{equation}\label{eq:local-continuation-DPP}
    V(t,x)=\E\!\left[V(t-\zeta_h,X_{\zeta_h}^x)\right].
\end{equation}
In particular, $\zeta_h<\sigma_0^x$ almost surely.
\end{proposition}

\subsubsection{Viscosity solution}
We first introduce the definition of viscosity solution in the current context. Then we prove the viscosity property and finally establish the comparison result for sub- and supersolutions.

\begin{definition}\label{def:viscosity}
A bounded continuous function $F$ on
$([0,\infty)\times(-\infty,0])\setminus\{(0,0)\}$ is an admissible viscosity solution of
\eqref{eq:VI-F}--\eqref{eq:VI-data} if:
\begin{enumerate}[label=\textup{(\roman*)}]
\item it satisfies the bounds \(p(t)H(t,x)\le F(t,x)\le\min\{p(t),H(t,x)\}\);
\item $F(t,0)=0$ for $t>0$, and $F(0,x)=1$  for
      $x<0$;
\item whenever $F-\varphi$ has a local maximum at an interior point,
\(\max\left\{\varphi_t-\frac12\varphi_{xx},\ F-p\right\}\le0\);
\item whenever $F-\varphi$ has a local minimum at an interior point,
\( \max\left\{\varphi_t-\frac12\varphi_{xx},\ F-p\right\}\ge0\).
\end{enumerate}
\end{definition}
\begin{remark}
As discussed before, we require no value condition on the point $(0,0)$ in the above definition. The limiting behavior near $(0,0)$ is controlled by the corresponding bounds.
\end{remark}
The following is the main result of this section.
\begin{theorem}\label{thm:wellposedness}
Assume that $p:[0,\infty)\to[0,1]$ is continuous and nonincreasing, with
$p(0)=1$.  Then the value function $V$ in \eqref{eq:backward-value} is the
unique viscosity solution in the sense of Definition~\ref{def:viscosity}.
Moreover:
\begin{enumerate}[label=\textup{(\roman*)}]
\item $x\mapsto V(t,x)$ is nonincreasing for every $t>0$;
\item $\lim_{x\to-\infty}V(t,x)=p(t)$;
\item for each $t$, the set $\{x\le 0:V(t,x)=p(t)\}$ is either empty or a closed ray
      $(-\infty,\hat b(t)]$;
\item the set
\[
    \cS_F:=\{(t,x):t>0,\ V(t,x)=p(t)\}
\]
      is relatively closed in $(0,\infty)\times(-\infty,0]$, and $\hat b$ is upper semicontinuous on $(0,\infty)$ as an extended-real-valued
      function. The first
      entry time \eqref{eq:tau-b} is a stopping time.
\end{enumerate}
\end{theorem}
The proof of this theorem will be delayed to Section \ref{sec:proof_vis}. Here we remark that it induces the optimal stopping rule in the auxiliary time-inversed stopping problem.
\begin{remark}
By Theorem~\ref{thm:wellposedness}, starting from $(t,x)$, stopping at time \(s\) is optimal precisely when $X_s^x\le \hat b(t-s)$.
Hence an optimal stopping time for the auxiliary problem is
$$
    \theta_{t,x}^*
    =
    \inf\{0\le s<t:X_s^x\le \hat b(t-s)\}
    \wedge\sigma_0^x\wedge t,
$$
namely the auxiliary problem encounters the boundary \(\hat b\) backward in time, through \(\hat b(t-s)\). The same boundary will be used forward in time in the inverse first-passage problem.
\end{remark}

\subsection{Recovery of the law}\label{sec:rocover_law}

In the previous section, Theorem \ref{thm:wellposedness} justifies the wellposedness of the variational inequality, but it remains open whether $\Pbb(\tau_{\hat b}>t)=p(t)$, namely $\tau_{\hat b}\sim\mu$.  For an upper-semicontinuous function $b:(0,\infty)\to[-\infty,0]$, define $b^-(t):=\limsup_{s\uparrow t}b(s)$ and introduce
\[
 Q_b:=\{(t,x):t>0,\ b(t)<x<0\}.
\]
We will also distinguish the closed and strict crossing times
\[
 \tau_b:=\inf\{t>0:Y_t\le b(t)\},\qquad
 \tau_b^\circ:=\inf\{t>0:Y_t<b(t)\}.
\]
All infima of empty sets are $+\infty$.

Although $\tau_b$ and $\tau_b^\circ$ need not coincide pathwise for a general boundary $b$, they agree almost surely for the boundary constructed above.

\begin{proposition}\label{prop:closed-strict-crossing}
For every upper-semicontinuous $b:(0,\infty)\to[-\infty,0]$,
$\tau_b=\tau_b^\circ$ almost surely.  In particular this holds for $\hat b$.
\end{proposition}

The proof is deferred to Section~\ref{sec:proof2}, where the rough boundary is approximated through a landmark scheme using ideas from \cite{CCCS}.  We now introduce the viscosity problem for a rough reflected barrier.
\begin{definition} \label{def:direct}
Fix $T>0$ and an upper-semicontinuous boundary $b$.  A continuous function
$G$ on $((0,T]\times(-\infty,0])\cup(\{0\}\times(-\infty,0))$ is a viscosity solution
associated with $b$ if:
\begin{enumerate}[label=(D\arabic*)]
\item $0\le G\le1$, $G(t,0)=0$, $G(0,x)=1$ for $x<0$, and
$\lim_{t\downarrow0}\sup_{x\le0}|G(t,x)-H(t,x)|=0$.
\item $x\mapsto G(t,x)$ is nonincreasing, and it is constant on
$(-\infty,b(t)]$ whenever $b(t)>-\infty$.
\item $G_t-\frac12G_{xx}=0$ in $Q_b\cap((0,T)\times(-\infty,0))$.
\item $G_t-\frac12G_{xx}\le0$ on $(0,T)\times(-\infty,0)$.
\end{enumerate}
The equation and inequality in (D3)--(D4) can be understood
in the viscosity sense. When $b(t)=-\infty$,
the constancy requirement is empty, and $G(t,b(t))$ should be understood as $\lim_{x\to-\infty}G(t,x)$.
\end{definition}

Introduce  $G^b(t,x):=\Pbb(\tau_b\ge t,\ Y_t>x)$. The following proposition is an adaptation of 
\cite[Theorem 3]{CCCS} in the current context, which identifies $G^b$ as the viscosity solution associated with boundary $b$.

\begin{proposition}\label{prop:direct}
Suppose $b:(0,\infty)\to[-\infty,0]$ is upper semicontinuous and, for
$0<t\le T$, $b(t)=b^-(t)$. In addition, $\sup_{t\in K}b(t)<0$ for every compact $K\subset(0,T]$.
Assume also $\Pbb(\tau_b>0)=1$.  Then 
 $G^b(t,x)$
is the unique viscosity solution associated with $b$ on $[0,T]$ in the sense of Definition \ref{def:direct}. 
Moreover, $\Pbb(\tau_b=t)=0$ for $0<t\le T$, and
\[
 \Pbb(\tau_b>t)=\lim_{x\to-\infty}G^b(t,x)=G^b(t,b(t)).
\]
\end{proposition}

On the above, the condition $b=b^-$ excludes upward jumps that could create an atom in the law of $\tau_b$.  The following is the main result of this section.
\begin{theorem}\label{thm:recovery_law}
Let $p:[0,\infty)\to[0,1]$ be continuous and nonincreasing, with $p(0)=1$.
Let $F$ be the solution in Theorem~\ref{thm:wellposedness}, and set
$\hat b(t):=\sup\{x\le0:F(t,x)=p(t)\}$.  Then
\[
 \Pbb(\tau_{\hat b}>t)=p(t),\qquad t\ge0.
\]
In addition,
\[
 F(t,x)=\Pbb(\tau_{\hat b}>t,Y_t>x),\qquad t>0,\ x\le0.
\]
Consequently $\tau_{\hat b}\sim\mu$ when $p(t)=\mu((t,\infty))$.
\end{theorem}

The proof permits $\hat b=-\infty$, including intervals on which the
target survival function is constant.  It also permits $p$ to vanish at
a finite time.  No continuity or bounded-variation assumption on
$\hat b$ is imposed.

\subsection{Optimality of the construction} \label{sec:optimality}

By Theorem~\ref{thm:recovery_law}, the boundary \(\hat b\) constructed above is upper semicontinuous, and it satisfies
\begin{align}\label{eq:optimal_stopping}
\hat\tau=\tau_{\hat b}=\inf\{t>0:Y_t\le\hat b(t)\}\sim\mu.
\end{align}
Upper semicontinuity makes the time-space stopping region $$\{(t,y) \in [0,\infty) \times (-\infty,0]: \, y \leq \hat b(t) \} $$ closed, so
$\hat\tau$ is a stopping time. Let \(\mathcal T\) denote the set of almost surely finite stopping times, and introduce
$$
    \mathcal T^\mu
    :=
    \{\rho\in\mathcal T:\rho\sim\mu\}.
$$
\begin{theorem}\label{thm:main}
For every stopping time $\rho \in \mathcal{T}^\mu$, $\E[M_\rho]\le \E[M_{\hat\tau}]$. 
Consequently,
\[
\sup_{\rho\in \mathcal{T}^\mu}\E[M_\rho]=\E[M_{\hat\tau}],
\]
where the value $+\infty$ is allowed. 
\end{theorem}

The proof of above theorem will be postponed to Section \ref{sec:proof3}. We will establish the result through discrete approximations and convergence arguments.  Compared with the approach of \cite{BEES2018}, our proof does not require the moment assumption on $\mu$. In Remark~\ref{rmk:duality}, let us explain intuitively why the solution of constrained optimal stopping is given by a hitting time of $(Y_t)$ as in \eqref{eq:optimal_stopping} from a duality point of view, and why the stopping time $\hat \tau$ is actually optimal for a class of objective functions.

\begin{remark} \label{rmk:duality}
 Similarly to \cite{GHLT, HLST2016}, we can tackle the constrained stopping problem using its dual formulation. Here is a formal derivation. Consider a class of objective functions of the form 
$$
U(x,m)=cm+h(x-m),
$$
where $c$ is a positive constant, $h \in C^2((-\infty,0])$ has a nondecreasing bounded second-order derivative, and that $h'(0) \leq c$. Then it holds that
\begin{align}\label{eq:general_objective}
    \sup_{\rho \in \mathcal{T}^{\mu}}\E[cM_\rho+h(B_{\rho}-M_\rho)]=\E[cM_{\hat \tau}+h(B_{\hat \tau}-M_{\hat \tau})],
\end{align}
whenever $\mu$ has a finite first moment. It includes the running maximum $U(x,m)=m$ when $c=1$ and $h\equiv 0$.

Suppose that the Fenchel-Moreau theorem holds, we have  that
\begin{align*}
    \sup_{\rho \in \mathcal{T}^\mu}\E[ cM_{\rho}+h(B_\rho-M_\rho)]&=\sup_{\rho \in \mathcal{T}} \inf_{\psi \in C_b(\R_+;\,\R)} \E[ cM_{\rho}+h(B_\rho-M_\rho)+\psi(\rho)]-\int_0^\infty \psi(t) \, \mu(dt) \\
    &= \inf_{\psi \in C_b(\R_+; \,\R)} D(\psi) -\int_0^\infty \psi(t) \, \mu(dt), 
\end{align*}
where $D(\psi)$ is an unconstrained optimal stopping problem
\begin{align*}
D(\psi):=\sup_{\rho \in \mathcal{T}}\E[cM_\rho+h(B_\rho-M_\rho)+\psi(\rho)].
\end{align*}
For each $\psi \in C_b(\R_+;\R)$, $D(\psi)- \int_0^\infty \psi(t) \,\mu(dt)$ provides an upper bound of the primal problem. We search for $\hat \psi$ so that the solution $\hat \tau$ to $D(\hat \psi)$ has distribution $\mu$. It is clear that such $(\hat \psi,\hat \tau)$ solves the constrained optimal stopping problem \eqref{eq:general_objective}.

For a dual  $\psi \in C_b(\R_+;\R)$, let
$u^\psi(t,x,m)$ be the value starting from $(B_t,M_t)=(x,m)$, with
$x\le m$,
\begin{align*}
    u^{\psi}(t,x,m)= \sup_{\rho \in \mathcal{T}, \, \, \rho \geq t} \E [ c M_{\rho} +h(B_\rho-M_\rho)+\psi(\rho) \,| \, (B_t,M_t)=(x,m)].
\end{align*}
Note that for $s \geq t$, $M_s - M_t= \left(B_t-M_t+ \sup_{u \leq s-t}(B_{t+u} -B_t) \right)^+$, and hence $M_s -M_t$, $Y_s-Y_t$ are independent of $\sigma(B_{t+u}-B_t: \, u \geq 0 )$ conditional on $Y_t$. Therefore we have the decomposition \(
    u^{\psi}(t,x,m)=cm+h(x-m)+ v^{\psi}(t,x-m)\), where 
\begin{align}\label{eq:vpsi}
    v^{\psi}(t,y):=\sup_{\eta} \E[c(y+W^*_\eta)^+ + h(y+W_\eta- (y+W^*_\eta)^+)-h(y) + \psi(t+\eta)], \quad y \leq 0.
\end{align}
Here $(W_t)$ is a Brownian motion independent of $(B_t)$ that starts with $0$, and $\eta$ ranges over all stopping times of $(W_t)$.

Set $L_s^y=(y+W_s^*)^+$ and $Y_s^y=y+W_s-L^y_s$. Since $dL_s^y$ is supported on $\{s:\, Y^y_s=0 \}$, It\^o's formula gives
\[
 \E\big[cL_s^y+h(Y_s^y)-h(y)\big]
 =(c-h'(0))\E[L_s^y]
   +\frac12\int_0^s\E\big[h''(Y_r^y)\big]\,dr.
\]
Note that $y \mapsto L^y_s$ and $y \mapsto Y^s_y$ are nondecreasing. Since $c \geq h'(0)$ and $y \mapsto h''(y)$ is nondecreasing, $$y \mapsto v^{\psi}(t,y), \qquad y \in (-\infty,0]$$ is  nondecreasing  for any fixed $t \geq 0$.

The immediate stopping yields $u^{\psi}(t,x,m) \geq cm +h(y) +\psi(t)$, which is equivalent to $v^\psi (t,x-m) \geq  \psi(t)$. 
As $y \mapsto v^\psi(t,y)$ is nondecreasing, the stopping region at time $t$ is given by $\{y\leq 0: v^\psi(t,y) \leq \psi(t)\}$.  Therefore the stopping boundary is given by 
\[
 b^\psi(t):=\sup\{y\le0:v^\psi(t,y)-\psi(t)=0\},
 \qquad \sup\varnothing:=-\infty,
\]
and an optimal stopping time for $D(\psi)$ is 
\[
 \tau^\psi:=\inf\{t\ge0:Y_t\le b^\psi(t)\}.
\]
It explains why the optimal stopping time is a hitting time of $(Y_t)$.

By the optimality of $\tau^{\psi}$,  for any $\rho \in \mathcal{T}$
\begin{align}
\E[cM_{\tau^\psi}+h(Y_{\tau^\psi})+\psi(\tau^{\psi})] \geq \E[cM_{\rho}+h(Y_\rho)+\psi(\rho)],
\end{align}
Furthermore if $\tau^{\psi} \sim \mu$ and $\rho \sim \mu$, we get that $\E[cM_{\tau^{\psi}}+h(Y_\tau^{\psi})] \geq \E[cM_{\rho}+h(Y_\rho)]$. Hence $\tau^{\psi}$ solves the constrained optimal stopping problem.

In Section~\ref{sec:proof3}, we prove the result for the case $U(x,m)=m$ using discrete approximation and stability arguments. The proof for general $U(x,m)$ will follow similar arguments.

\end{remark}

\section{Proof of the well-posedness theorem}\label{sec:proof1}

\subsection{Property of the time-reversed stopping problem}

\begin{proof}[Proof of Lemma~\ref{lem:corner-bounds}]
Immediate stopping implies that $V(t,x)\le p(t)$.  Choosing
$\theta=t\wedge\sigma_0^x$, we have
    $p(t-\theta) \1_{\{\theta<\sigma_0^x\}}
    =\1_{\{t<\sigma_0^x\}}$,
hence $V(t,x)\le H(t,x)$.

For the lower bound, fix any admissible $\theta\le t$.  On $\{\sigma_0^x>t\}$ we have $\theta<\sigma_0^x$; since $p$ is nonincreasing, we have
    $p(t-\theta)\ge p(t)$.
Therefore
\[
    \E\left[p(t-\theta) \1_{\{\theta<\sigma_0^x\}}\right]
    \ge p(t)\Pbb(\sigma_0^x>t)
    =p(t)H(t,x).
\]
Taking the infimum over $\theta$ proves \eqref{eq:corner-bounds}.
\end{proof}

\begin{proof}[Proof of Proposition \ref{lem:local-continuation-DPP}]
Fix $\varepsilon>0$. By the strong Markov property and measurable
selection of $\varepsilon$-optimal stopping rules for the optimization
problem at time $\rho$, there exists $\theta_\varepsilon\in\cT$
such that
\[
    \rho\le\theta_\varepsilon\le t,
    \qquad
    \E\!\left[
        p(t-\theta_\varepsilon)
        \1_{\{\theta_\varepsilon<\sigma_0^x\}}
        \,\middle|\,\mathcal F_\rho
    \right]
    \le V(t-\rho,X_\rho^x)+\varepsilon
    \quad\text{on }\{\rho<\sigma_0^x\}.
\]
Since $\theta_\varepsilon\ge\rho$, the conditional expectation
vanishes on $\{\rho\ge\sigma_0^x\}$. Hence
\[
    V(t,x)
    \le
    \E\!\left[
        p(t-\theta_\varepsilon)
        \1_{\{\theta_\varepsilon<\sigma_0^x\}}
    \right] 
    \le
    \E\!\left[
        V(t-\rho,X_\rho^x)\1_{\{\rho<\sigma_0^x\}}
    \right]
    +\varepsilon\,\Pbb(\rho<\sigma_0^x).
\]
Letting $\varepsilon\to0$ yields
\eqref{eq:stopped-waiting-inequality}.

For the second assertion, by the compactness of $\overline Q$ and continuity of
$V-p$, there exists a number $\delta_Q>0$ such that
$V(q,x)\leq p(q)-\delta_Q$, for all $(q,x)\in\overline Q$. Fix $0<h<t$ and let $\zeta_h:=h\wedge\nu_Q$. By continuity
of the paths and the assumptions on $Q$, we have almost surely, $ \zeta_h<\sigma_0^x$, and $(t-s,X_s^x)\in\overline Q$ for $0\le s\le\zeta_h$. 
Consequently,
\[
    p(t-s)-V(t-s,X_s^x)\ge\delta_Q>0,
    \qquad 0\le s\le\zeta_h.
\]
The martingale property of the Snell envelope in the continuation
region therefore implies that
\[
    \left(
        V\bigl(t-(s\wedge\zeta_h),X_{s\wedge\zeta_h}^x\bigr)
    \right)_{0\le s\le h}
    \quad\text{is a bounded martingale}.
\]
Optional sampling yields
    $V(t,x)
    =\E\!\left[V(t-\zeta_h,X_{\zeta_h}^x)\right]$, 
which proves \eqref{eq:local-continuation-DPP}.
\end{proof}

\subsection{Proof of Theorem~\ref{thm:wellposedness}}\label{sec:proof_vis}

Throughout this subsection, $p:[0,\infty)\to[0,1]$ is continuous and
nonincreasing, with $p(0)=1$, and $V$ is defined by
\eqref{eq:backward-value}. We establish four auxiliary results and then
combine them to prove Theorem~\ref{thm:wellposedness}.

We start with the continuity of $V$.
\begin{lemma}\label{lem:V-continuity}
The value function $V$ is continuous on
$([0,\infty)\times(-\infty,0])\setminus\{(0,0)\}$, with
\[
 V(t,0)=0\quad(t>0),\qquad V(0,x)=1\quad(x<0).
\]
\end{lemma}

\begin{proof}
Fix $T>0$ and let
\[
\omega_T(h):=
\sup_{\substack{u,v\in[0,T]\\ |u-v|\le h}}
|p(u)-p(v)|.
\]
Since $p$ is continuous, $\omega_T(h)\to0$ as $h\to0$.

We first prove continuity in time. Fix $x<0$ and
$0<t<t+h\le T$. Since $p$ is nonincreasing and every stopping
time bounded by $t$ is also bounded by $t+h$, we have
$V(t+h,x)\le V(t,x)$.
Conversely, let $\theta\le t+h$ be a stopping time and let
$\bar\theta:=\theta\wedge t$. Then, pathwise,
\[
p(t-\bar\theta)\mathbf1_{\{\bar\theta<\sigma_0^x\}}
\le
p(t+h-\theta)\mathbf1_{\{\theta<\sigma_0^x\}}
+\omega_T(h)
+\mathbf1_{\{t<\sigma_0^x\le t+h\}}.
\]
Indeed, on $\{\theta\le t\}$ the arguments of $p$ differ by $h$;
on $\{t<\theta<\sigma_0^x\}$, the difference between
the two payoffs is
$1-p(t+h-\theta)\le\omega_T(h)$;
on $\{t<\sigma_0^x\le\theta\}$, the l.h.s. is bounded by 
$\mathbf1_{\{t<\sigma_0^x\le t+h\}}$.
Taking expectations and then the infimum over $\theta\le t+h$, we obtain
\[
0\le V(t,x)-V(t+h,x)
\le
\omega_T(h)+H(t,x)-H(t+h,x),
\]
where $H(t,x)=\mathbb P(\sigma_0^x>t)$.
The continuity of $H(\cdot,x)$ on $(0,\infty)$ implies that
$V(\cdot,x)$ is continuous on $(0,\infty)$.

We next use the scaling property of Brownian motion to obtain joint continuity in the
interior. For $c>0$, scaling time by $c$, space by $\sqrt c$,
and the admissible stopping times accordingly yields
\[
V(ct,\sqrt c\,x)
=
\inf_{\substack{\theta\in\mathcal T\\0\le\theta\le t}}
\mathbb E\left[
p\bigl(c(t-\theta)\bigr)
\mathbf1_{\{\theta<\sigma_0^x\}}
\right].
\]
Consequently, whenever $t,ct\le T$, we have
$|V(ct,\sqrt c\,x)-V(t,x)|
\le \omega_T(|c-1|t)$.
Let $(t_n,x_n)\to(t,x)$ with $t>0$ and $x<0$, and let
$c_n:=\frac{x_n^2}{x^2}$, 
$s_n:=\frac{t_n}{c_n}$.
For all sufficiently large $n$, we have $x_n<0$ and
$(t_n,x_n)=(c_ns_n,\sqrt{c_n}\,x)$, 
$c_n\to1$, $s_n\to t$.
Choosing $T$ large enough that $t_n,s_n\le T$ for all sufficiently
large $n$, we obtain
\[
\begin{aligned}
|V(t_n,x_n)-V(t,x)|
&\le
|V(c_ns_n,\sqrt{c_n}\,x)-V(s_n,x)|
+|V(s_n,x)-V(t,x)|\\
&\le
\omega_T(|t_n-s_n|)
+|V(s_n,x)-V(t,x)|
\longrightarrow0.
\end{aligned}
\]
Thus $V$ is jointly continuous on
$(0,\infty)\times(-\infty,0)$.

Finally, by Lemma~\ref{lem:corner-bounds}, we have
$
p(t)H(t,x)\le V(t,x)\le H(t,x)$, and
$H(t,x)=2\Phi\left(\frac{-x}{\sqrt t}\right)-1$.
For $0<\varepsilon<T$, it follows that
\[
\sup_{t\in[\varepsilon,T]}|V(t,x)|
\le
\sqrt{\frac{2}{\pi\varepsilon}}\, -x)
\longrightarrow0,
\quad \mbox{ when } x\uparrow0.
\]
Moreover, for every $d>0$,
\[
\sup_{x\le-d}|1-V(t,x)|
\le
1-p(t)H(t,-d)
\le
1-p(t)+2\left(1-\Phi\left(\frac d{\sqrt t}\right)\right)
\longrightarrow0,
\quad \mbox{ when } t\to0.
\]
By definition, $V(t,0)=0$ for $t>0$ and $V(0,x)=1$
for $x<0$. These estimates establish joint continuity at every
boundary point other than $(0,0)$ and complete the proof.
\end{proof}

In the following proposition, we establish the viscosity solution property.

\begin{proposition}\label{prop:V-viscosity}
The value function $V$ is a viscosity subsolution and supersolution of
\eqref{eq:VI-F} on $(0,\infty)\times(-\infty,0)$.
\end{proposition}

\begin{proof}
\noindent\emph{Subsolution property.}
Let $\varphi\in C^{1,2}$ touch $V$ from above at an interior point
$(t_0,x_0)$. Suppose without loss of generality that 
    $(V-\varphi)(t_0,x_0)=0$. 
Choose a bounded open cylinder $Q$, with compact closure, such that
    $(t_0,x_0)\in Q$, and
 $\overline Q\subset(0,\infty)\times(-\infty,0)$, 
 $V\leq\varphi$ on $\overline Q$.
We further define
    $\nu_Q:=\inf\{s\geq0:(t_0-s,X_s^{x_0})\notin Q\}$, 
    $\zeta_h:=h\wedge\nu_Q$. 
For all sufficiently small $h>0$, Proposition
\ref{lem:local-continuation-DPP} and the fact that
$\zeta_h<\sigma_0^{x_0}$ imply
\[
    \varphi(t_0,x_0)
    =V(t_0,x_0)
    \leq \E\!\left[
       V(t_0-\zeta_h,X_{\zeta_h}^{x_0})\right] 
    \leq \E\!\left[
       \varphi(t_0-\zeta_h,X_{\zeta_h}^{x_0})\right].
\]
Applying It\^o's formula to the stopped process yields
\begin{equation}\label{eq:subsolution-stopped-Ito}
    0\leq
    \E\!\left[\int_0^{\zeta_h}
       \left(-\varphi_t+\frac12\varphi_{xx}\right)
       (t_0-s,X_s^{x_0})\,ds\right].
\end{equation}
Since $(t_0,x_0)\in Q$ and the paths of
$s\mapsto(t_0-s,X_s^{x_0})$ are continuous, we have
$\nu_Q>0$ almost surely. Hence
$\zeta_h/h\to1$ almost surely as $h\to0$.
By continuity of the integrand,
\[
\frac1h\int_0^{\zeta_h}
\left(-\varphi_t+\frac12\varphi_{xx}\right)
(t_0-s,X_s^{x_0})\,ds
\longrightarrow
\left(-\varphi_t+\frac12\varphi_{xx}\right)(t_0,x_0)
\quad\text{almost surely}.
\]
Since the integrand is bounded on $\overline Q$ and
$\zeta_h\le h$, dominated convergence allows us to divide
\eqref{eq:subsolution-stopped-Ito} by $h$ and let
$h\to0$, yielding
   $-\varphi_t(t_0,x_0)
       +\frac12\varphi_{xx}(t_0,x_0)\geq0$,
or equivalently
    $\varphi_t(t_0,x_0)
       -\frac12\varphi_{xx}(t_0,x_0)\leq0$. 
Immediate stopping gives $V(t_0,x_0)\leq p(t_0)$; hence
\[
    \max\left\{
       \varphi_t(t_0,x_0)-\frac12\varphi_{xx}(t_0,x_0),
       V(t_0,x_0)-p(t_0)
    \right\}\leq0.
\]

\smallskip
\noindent\emph{Supersolution property.}
Let $\varphi\in C^{1,2}$ touch $V$ from below at an interior point
$(t_0,x_0)$, and normalize it so that
$(V-\varphi)(t_0,x_0)=0$.  If
$V(t_0,x_0)=p(t_0)$, then
\[
    \max\left\{
       \varphi_t(t_0,x_0)-\frac12\varphi_{xx}(t_0,x_0),
       V(t_0,x_0)-p(t_0)
    \right\}\geq0
\]
follows immediately from the obstacle branch.

Suppose now that $V(t_0,x_0)<p(t_0)$, and set
$\mathcal C:=\{(t,x)\in(0,\infty)\times(-\infty,0):V(t,x)<p(t)\}$.  Since $V-p$ is continuous, there is
an open cylinder $Q$ such that
    $(t_0,x_0)\in Q$,
    $\overline Q\subset\mathcal C$, and
    $V\geq\varphi$ on $\overline Q$, 
with $\nu_Q$ and $\zeta_h=h\wedge\nu_Q$ defined as above. The continuation
identity \eqref{eq:local-continuation-DPP} implies
\[
    \varphi(t_0,x_0)
    =V(t_0,x_0)
    =\E\!\left[
       V(t_0-\zeta_h,X_{\zeta_h}^{x_0})\right] 
    \geq\E\!\left[
       \varphi(t_0-\zeta_h,X_{\zeta_h}^{x_0})\right].
\]
By It\^o's formula, we have
\[
    0\geq
    \E\!\left[\int_0^{\zeta_h}
       \left(-\varphi_t+\frac12\varphi_{xx}\right)
       (t_0-s,X_s^{x_0})\,ds\right].
\]
As in the subsolution argument, $\E[\zeta_h]/h\to1$.  Dividing by $h$ and letting
$h\to0$ yields
    $-\varphi_t(t_0,x_0)
       +\frac12\varphi_{xx}(t_0,x_0)\leq0$,
and hence
\[
    \varphi_t(t_0,x_0)
       -\frac12\varphi_{xx}(t_0,x_0)\geq0.
\]
Thus the viscosity supersolution inequality holds in the continuation region
as well.
\end{proof}

In the next proposition, we establish the comparison principle. The proof handles the singularity at $(0,0)$ by first establishing
comparison from a strictly positive initial time and then letting
that time tend to zero.

\begin{proposition}\label{prop:V-comparison}
Let $u$ be a bounded upper semicontinuous subsolution and $v$ be a
bounded lower semicontinuous supersolution of \eqref{eq:VI-F} on
$(0,\infty)\times(-\infty,0)$. Suppose that
\begin{equation}\label{eq:cmp4-bounds}
u(t,x),v(t,x)
  \in [ p(t)H(t,x), \, \min\{p(t),H(t,x)\}],\qquad t>0,\ x<0.
\end{equation}
Then $u\leq v$ on $(0,\infty)\times(-\infty,0)$. Consequently, there is at most one viscosity solution in the sense of
Definition~\ref{def:viscosity}.
\end{proposition}

\begin{proof}
Extend $u$ and $v$ to $x=0$ by setting
$u(t,0)=v(t,0)=0$ for all $t>0$. 
For every $0<\varepsilon<T$, $t\in[\varepsilon,T]$ and $x<0$, we have
\[
 \max\{ u(t,x),v(t,x)\}\le H(t,x)
\le \sqrt{\frac{2}{\pi\varepsilon}}\,(-x).
\]
Thus the extensions remain upper and lower semicontinuous respectively.

\medskip
\noindent\emph{Step 1: reduction to positive times.}
Fix $\varepsilon>0$ and let $d_\varepsilon:=1-p(\varepsilon)$. 
\eqref{eq:cmp4-bounds} implies that for $x\le0$,
$$
u(\varepsilon,x)-v(\varepsilon,x)
\le H(\varepsilon,x)-p(\varepsilon)H(\varepsilon,x)
\le d_\varepsilon.
$$
The aim of the remaining proof is to extend the above inequality to $t > \varepsilon$, i.e. for  $t\ge\varepsilon$, $x\le0$, we have
\begin{equation}\label{eq:cmp4-strip}
u(t,x)-v(t,x)\le d_\varepsilon.
\end{equation}
Fix $\kappa>0$ and define $w(t,x):=u(t,x)-d_\varepsilon-\kappa(t-\varepsilon)$. Then
\[
w(\varepsilon,x)\le v(\varepsilon,x),
\qquad
w(t,0)=-d_\varepsilon-\kappa(t-\varepsilon)\le v(t,0),
\qquad
w(t,x)\le p(t).
\]
If $\varphi$ touches $w$ from above at $(t,x)$, for
$t>\varepsilon$ and $x<0$, let 
$\psi(r,y):=\varphi(r,y)+d_\varepsilon+\kappa(r-\varepsilon)$.
Since $u-\psi=w-\varphi$, the function $\psi$ touches $u$ from above at the same point.
The subsolution inequality implies 
\[
0\ge
\max\left\{
\psi_t-\frac12\psi_{xx},\,u-p(t)
\right\}
\ge
\varphi_t+\kappa-\frac12\varphi_{xx}.
\]
Hence
\begin{equation}\label{eq:cmp4-strict}
\varphi_t-\frac12\varphi_{xx}\le-\kappa.
\end{equation}

\medskip
\noindent\emph{Step 2: localization and doubling of variables.}
Suppose that
$g:=w(t_*,x_*)-v(t_*,x_*)>0$. 
The boundary inequalities imply $t_*>\varepsilon$ and $x_*<0$.
Choose $T>t_*$ and then $\eta>0$ such that
$2\eta x_*^2+\frac{\eta}{T-t_*}<\frac g2$ and $2\eta+4\eta^2<\kappa$.
Let $m:=g/2$. For $0<\alpha,\beta\le1$, $\varepsilon\le s,t<T$, 
and $x,z\le0$, define
\[
P_{\alpha,\beta}(t,x,s,z)
:=
\frac{(x-z)^2}{2\alpha}
+\frac{(t-s)^2}{2\beta}
+\eta(x^2+z^2)
+\frac{\eta}{T-t},\ \ 
\Psi_{\alpha,\beta}(t,x,s,z)
:=w(t,x)-v(s,z)-P_{\alpha,\beta}(t,x,s,z).
\]
At the chosen point, $\Psi_{\alpha,\beta}(t_*,x_*,t_*,x_*)
=
g-2\eta x_*^2-\frac{\eta}{T-t_*}>m$. 
Choose $M>0$ such that $|w|,|v|\le M$ on
$[\varepsilon,T]\times(-\infty,0]$.
Whenever $\Psi_{\alpha,\beta}\ge m$, we have
\[
P_{\alpha,\beta}
\le w(t,x)-v(s,z)-m
\le2M.
\]
In particular, 
$x^2+z^2\le\frac{2M}{\eta}$, and
$T-t\ge\frac{\eta}{2M}$. 
Upper semicontinuity therefore yields a maximizer
$(t_{\alpha,\beta},x_{\alpha,\beta},
 s_{\alpha,\beta},z_{\alpha,\beta})$,
at which
\[
\Psi_{\alpha,\beta}>m,
\qquad
|x_{\alpha,\beta}-z_{\alpha,\beta}|^2\le4M\alpha,
\qquad
|t_{\alpha,\beta}-s_{\alpha,\beta}|^2\le4M\beta.
\]

Along any sequence $\alpha,\beta\to0$, a subsequence
of the maximizers converges to $
(\bar t,\bar x,\bar t,\bar x)$, with
$\varepsilon\le\bar t<T$, $\bar x\le0$.
By semicontinuity and nonnegativity of the quadratic penalties,
\[
m\le
w(\bar t,\bar x)-v(\bar t,\bar x)
-2\eta\bar x^2-\frac{\eta}{T-\bar t}.
\]
If $\bar t=\varepsilon$, the right-hand side is at most 
$-2\eta\bar x^2-\frac{\eta}{T-\varepsilon}<0$. 
If $\bar x=0$, it equals to
$-d_\varepsilon-\kappa(\bar t-\varepsilon)
-\frac{\eta}{T-\bar t}<0$.
Both cases are impossible. Consequently, for all sufficiently
small $\alpha,\beta$, both maximizing points lie in
$
(\varepsilon,T)\times(-\infty,0)$.

At these maximizers, abbreviate the coordinates by $(t,x,s,z)$.
Since $P_{\alpha,\beta}\ge0$, we have 
$w(t,x)-v(s,z)>m$.
Using $w(t,x)\le p(t)$, we obtain
$v(s,z)-p(s)
<
p(t)-p(s)-m$.
Let
\[
\omega_{\varepsilon,T}(r)
:=
\sup_{\substack{a,b\in[\varepsilon,T]\\|a-b|\le r}}
|p(a)-p(b)|.
\]
Then
$v(s,z)-p(s)
<
\omega_{\varepsilon,T}(|t-s|)-m$. 
Since $|t-s|\to0$, for all sufficiently small $\alpha,\beta$, we have
\begin{equation}\label{eq:cmp4-active}
v(s,z)-p(s)<-\frac m2<0.
\end{equation}

\medskip
\noindent\emph{Step 3: jet inequalities and contradiction.}
Apply the elliptic theorem of sums
\cite[Theorem~3.2]{CIL1992} to $w$ and $-v$,
regarding $(t,x)$ and $(s,z)$ as two-dimensional
Euclidean variables, and choose the parameter of that theorem to be $\alpha$.
There exist symmetric $2\times2$ matrices $B,C$ such that
\[
\bigl(D_{(t,x)}P_{\alpha,\beta},B\bigr)
\in\overline J^{2,+}w(t,x),
\qquad
\bigl(-D_{(s,z)}P_{\alpha,\beta},C\bigr)
\in\overline J^{2,-}v(s,z),
\]
and
\[
\begin{pmatrix}
B&0\\
0&-C
\end{pmatrix}
\le A+\alpha A^2,
\qquad
A:=D^2P_{\alpha,\beta}.
\]
Here the jets are taken in the space--time variables.
The viscosity inequalities apply to these closed jets:
$C^2$ space--time test functions are $C^{1,2}$, and the
inequalities pass to the jet limits.

Denote 
$X:=B_{22}$ and $Y:=C_{22}$.
The time components of the jet gradients are
\begin{equation}\label{eq:cmp4-time}
\begin{aligned}
a&:=\partial_tP_{\alpha,\beta}
=\frac{t-s}{\beta}+\frac{\eta}{(T-t)^2},\\
b&:=-\partial_sP_{\alpha,\beta}
=\frac{t-s}{\beta},\\
a-b&=\frac{\eta}{(T-t)^2}>0.
\end{aligned}
\end{equation}

In coordinate order $(t,x,s,z)$, we have
\[
A=
\begin{pmatrix}
\beta^{-1}+2\eta(T-t)^{-3}&0&-\beta^{-1}&0\\
0&\alpha^{-1}+2\eta&0&-\alpha^{-1}\\
-\beta^{-1}&0&\beta^{-1}&0\\
0&-\alpha^{-1}&0&\alpha^{-1}+2\eta
\end{pmatrix}.
\]
For $e=(0,1,0,1)^\top$, we have
\[
Ae=(0,2\eta,0,2\eta)^\top,
\qquad
e^\top Ae=4\eta,
\qquad
e^\top A^2e=|Ae|^2=8\eta^2.
\]
Testing the matrix inequality against $e$ yields
\begin{equation}\label{eq:cmp4-matrix}
\begin{aligned}
X-Y
=
e^\top
\begin{pmatrix}
B&0\\
0&-C
\end{pmatrix}e
&\le e^\top(A+\alpha A^2)e
=4\eta+8\alpha\eta^2.
\end{aligned}
\end{equation}

By \eqref{eq:cmp4-strict}, we have 
$a-\frac12X\le-\kappa$. 
The supersolution property and \eqref{eq:cmp4-active} imply that
\[
\max\left\{
b-\frac12Y,\,
v(s,z)-p(s)
\right\}\ge0,
\qquad
v(s,z)-p(s)<0,
\]
and hence $b-\frac12Y\ge0$. 
Subtracting these inequalities yields
\(
a-b\le\frac12(X-Y)-\kappa
\). 
Combining this with
\eqref{eq:cmp4-time}--\eqref{eq:cmp4-matrix}, we obtain
\[
0<
\frac{\eta}{(T-t)^2}
=a-b
\le\frac12(X-Y)-\kappa
\le2\eta+4\alpha\eta^2-\kappa
\le2\eta+4\eta^2-\kappa
<0,
\]
which is a contradiction. Therefore $w\le v$.

\medskip
\noindent\emph{Step 4: removal of the perturbations.}
We have proved that, for every $\kappa>0$, $t\ge\varepsilon$, $x\le0$, it holds that
$u(t,x)-v(t,x)
\le d_\varepsilon+\kappa(t-\varepsilon)$. 
Letting $\kappa\to0$, we get \eqref{eq:cmp4-strip}. For each fixed $t>0$ and $x\le0$, we may then let
$\varepsilon\to0$, with $\varepsilon<t$, and then obtain
\[
u(t,x)-v(t,x)
\le\lim_{\varepsilon\to0}
\bigl(1-p(\varepsilon)\bigr)=0.
\]
The initial values  of $u,v$ agree for $t=0$, $x<0$. Applying comparison in both directions, we obtain the uniqueness.
\end{proof}

We finally establish the spatial monotonicity of $V$, the semicontinuity
of the associated boundary, and the stopping-time property of its first
entry time.

\begin{lemma}\label{lem:V-contact-boundary}
For every $t>0$, $x\mapsto V(t,x)$ is nonincreasing and
$\lim_{x\to-\infty}V(t,x)=p(t)$.
Let $\hat b(t):=\sup\{x\le0:V(t,x)=p(t)\}$, with $\sup\varnothing=-\infty$.
For each $t>0$, this set is empty or $(-\infty,\hat b(t)]$.
The stopping region
\[
 \mathcal S:=\{(t,x)\in(0,\infty)\times(-\infty,0]:V(t,x)=p(t)\}
\]
is relatively closed, and $\hat b$ is upper semicontinuous on $(0,\infty)$.
For any nonpositive continuous adapted process $Y$ on a filtered probability
space satisfying the usual conditions,
$\tau_{\hat b}=\inf\{t>0:Y_t\le \hat b(t)\}$ is a stopping time.
\end{lemma}

\begin{proof}
For $x_1<x_2\le0$, use the same Brownian motion in $X_s^{x_i}=x_i+W_s$.
Then $\sigma_0^{x_1}\ge\sigma_0^{x_2}$, so, for every $0\le\theta\le t$,
\[
 \E[p(t-\theta)\1_{\{\theta<\sigma_0^{x_1}\}}]
 \ge \E[p(t-\theta)\1_{\{\theta<\sigma_0^{x_2}\}}].
\]
Taking infima gives $V(t,x_1)\ge V(t,x_2)$.
Moreover,
\[
 0\le p(t)-V(t,x)\le p(t)(1-H(t,x))\longrightarrow0
 \mbox{ when } x\to-\infty.
\]
Fix $t>0$ and write $S_t=\{x\le0:V(t,x)=p(t)\}$.
Continuity makes $S_t$ closed. If $x\in S_t$ and $z\le x$, then
$p(t)\ge V(t,z)\ge V(t,x)=p(t)$, hence $z\in S_t$.
If $S_t\ne\varnothing$, choose $x_j\in S_t$ with $x_j\to \hat b(t)$;
continuity gives $\hat b(t)\in S_t$ and $S_t=(-\infty,\hat b(t)]$.
Otherwise $\hat b(t)=-\infty$.

Joint continuity implies the relative closedness of $\mathcal S$. If $t_j\to t>0$ and $\ell=\limsup_j\hat b(t_j)>-\infty$, pass to a
subsequence with $\hat b(t_j)>-\infty$ and $\hat b(t_j)\to\ell$. Then $(t_j,\hat b(t_j))\in\mathcal S$ implies $(t,\ell)\in\mathcal S$, so
$\ell\le \hat b(t)$. If $\ell=-\infty$, this inequality is automatic. Thus $\hat b(t)\ge\limsup_j\hat b(t_j)$.

For the stopping-time assertion, set $Z_t=p(t)-V(t,Y_t)$ for $t>0$.
This is a nonnegative continuous adapted process, and
$\{Z_t=0\}=\{Y_t\le \hat b(t)\}$.
Under the usual conditions, exceptional null sets for the continuous
versions belong to $\mathcal F_0$.
Let $\tau_n=\inf\{s\ge1/n:Z_s=0\}$.
For $t\ge1/n$,
\[
 \{\tau_n\le t\}
 =\left\{\min_{s\in[1/n,t]}Z_s=0\right\}\in\mathcal F_t.
\]
Indeed the minimum equals the infimum over $(\mathbb Q\cap[1/n,t])\cup\{1/n,t\}$; for $t<1/n$ the event is empty.
Thus each $\tau_n$ is a stopping time, and $\tau_n\to\tau_{\hat b}$,  $\{\tau_{\hat b}<t\}=\bigcup_n\{\tau_n<t\}\in\mathcal F_t$.
Now by the right-continuity of the filtration, we have
\[
 \{\tau_{\hat b}\le t\}=\bigcap_{m\ge1}\{\tau_{\hat b}<t+1/m\}
 \in\bigcap_{m\ge1}\mathcal F_{t+1/m}=\mathcal F_t.
 \vspace{-8mm}
\]
\end{proof}

\vspace{2mm}
\begin{proof}[Proof of Theorem~\ref{thm:wellposedness}]
Combining Lemma~\ref{lem:V-continuity}, Proposition~\ref{prop:V-viscosity}, Definition~\ref{def:viscosity} and Lemma~\ref{lem:V-contact-boundary}, we can draw the conclusion.
\end{proof}

\section{Proof of the law recovery} \label{sec:proof2}

\subsection{Landmark approximations and characterization of the surviving distribution}

We first prove Proposition~\ref{prop:closed-strict-crossing}.
\begin{proof}[Proof of Proposition~\ref{prop:closed-strict-crossing}]
Let \(W\) be a standard Brownian motion starting from zero and let \(\widetilde Y_t:=-|W_t|\). Since \(Y\) and \(\widetilde Y\) have the
same law on continuous path space, it suffices to prove the assertion
with \(Y\) replaced by \(\widetilde Y\). For \(\varepsilon\in\{-1,1\}\), define
\[
 \sigma_\varepsilon
 :=\inf\{t>0:\varepsilon W_t\le b(t)\},
 \qquad
 \sigma_\varepsilon^\circ
 :=\inf\{t>0:\varepsilon W_t<b(t)\},
\]
with \(\inf\varnothing=\infty\). Because \(b(t)\le0\), for every \(t>0\), we have
\[
 \{\widetilde Y_t\le b(t)\}
   =\{W_t\le b(t)\}\cup\{-W_t\le b(t)\}, \ \ 
 \{\widetilde Y_t<b(t)\}
   =\{W_t<b(t)\}\cup\{-W_t<b(t)\}.
\]
Taking the infimum over \(t>0\), we have
 $\tau_b(\widetilde Y)=\sigma_1\wedge\sigma_{-1}$ 
and $\tau_b^\circ(\widetilde Y)
       =\sigma_1^\circ\wedge\sigma_{-1}^\circ$.
Each process \(\varepsilon W\) is a standard Brownian motion. Applying
\cite[Proposition~2]{CCCS} to the boundary \(b\) for each value of
\(\varepsilon\), we obtain
 $\mathbb P(\sigma_\varepsilon
              =\sigma_\varepsilon^\circ)=1$, for
$\varepsilon\in\{-1,1\}$.
Hence
\[
 \begin{aligned}
 \mathbb P\bigl(
   \tau_b(\widetilde Y)\ne\tau_b^\circ(\widetilde Y)
 \bigr)
 &\le
 \sum_{\varepsilon\in\{-1,1\}}
 \mathbb P(\sigma_\varepsilon
              \ne\sigma_\varepsilon^\circ)
 =0.
 \end{aligned}
\]
The equality in law of \(Y\) and \(\widetilde Y\) now yields
\(\tau_b(Y)=\tau_b^\circ(Y)\) almost surely.
\end{proof}

In the following, we adapt the landmark construction of \cite{CCCS} to the
upper-semicontinuous lower boundary $b$. For the selection of landmarks, we can extend $b$ by $b(0)=0$.
 We introduce $I_i^n=[i2^{-n},(i+1)2^{-n}]$ for $n\ge1$ and $i\ge0$, and define
$b(t_i^n)=\sup_{s\in I_i^n}b(s)$ for $t_i^n\in I_i^n$. 
The maximum is attained by upper semicontinuity; and if $b\equiv-\infty$ on $I_i^n$, we choose its left endpoint. Define
\[
 \A_n(b):=\left(\bigcup_{m=1}^n\{t_i^m:i\ge0\}\right)\cap(0,\infty),
 \qquad
 \tau_b^{(n)}:=\inf\{t\in\A_n(b):Y_t\le b(t)\}.
\]
The intersection with $(0,\infty)$ ensures that the condition $Y_t\le b(t)$ is checked only at strictly positive times,
as in the definition of $\tau_b$. Since $\A_n(b)\subseteq\A_{n+1}(b)$, we have
$\tau_b^{(n+1)}\le\tau_b^{(n)}$.

The monotonicity of \(\tau_b^{(n)}\) alone does not identify its limit. We use the equality of weak and strict crossing times proved
above to show that the landmark observations recover \(\tau_b\). This also allows us to pass from the landmark functions \(G_n\) to
\(G^b\).

\begin{lemma}\label{lem:landmark}
For every upper-semicontinuous $b:(0,\infty)\to[-\infty,0]$,
\[
 \tau_b^{(n)}\downarrow\tau_b\quad\hbox{almost surely}.
\]
Consequently, for every $t>0$ and $x\le0$,
\[
 G_n(t,x):=\Pbb(\tau_b^{(n)}\ge t,Y_t>x)
 \longrightarrow G^b(t,x):=\Pbb(\tau_b\ge t,Y_t>x).
\]
\end{lemma}
\begin{proof}
First, we have $\tau_b\le\lim_n\tau_b^{(n)}$ pathwise.  If a continuous path satisfies $Y_r<b(r)$ at some $r>0$, choose its dyadic landmark $t_i^n$ in a cell
containing $r$.  Then $t_i^n\to r$ and $b(t_i^n)\ge b(r)$, so $Y_{t_i^n}<b(t_i^n)$ for all sufficiently large $n$.  These landmarks
are positive for large $n$.  Taking the infimum over such $r$, we have
$\lim_n\tau_b^{(n)}\le\tau_b^\circ$.  Now by Proposition \ref{prop:closed-strict-crossing} we can draw the conclusion the first assertion.  The second assertion
follows directly from
\[
 0\le G_n(t,x)-G^b(t,x)
 \le\Pbb(\tau_b^{(n)}\ge t>\tau_b)\longrightarrow0.
  \vspace{-8mm}
\]
\end{proof}

 \vspace{2mm}

We remind here how $G_n$ evolves between landmark times and how it changes at a landmark. Let \(\mathsf K_h\) denote the semigroup of Brownian motion
absorbed at zero on \((-\infty,0)\):
\[
 (\mathsf K_h f)(x)
 =\int_{-\infty}^0
   \bigl(g_h(z-x)-g_h(z+x)\bigr)f(z)\,dz,
 \qquad
 g_h(z):=(2\pi h)^{-1/2}e^{-z^2/(2h)}.
\]
The transition density of the reflected process \(Y\) is $q_h(y,z)=g_h(z-y)+g_h(z+y)$, for $y,z \le0$.
If \(\nu\) is a finite measure on \((-\infty,0]\) and
\(f(x):=\nu((x,0])\), then Fubini's theorem implies that for $x<0$, we have
\begin{equation}\label{eq:reflected-absorbed-tail}
 \int_{(-\infty,0]}\nu(dy)\int_x^0q_h(y,z)\,dz
   =(\mathsf K_h f)(x). 
\end{equation}
Thus, when a distribution is propagated by the reflected process, the function \(x\mapsto\nu((x,0])\) evolves according to \(\mathsf K\). In particular, before the first
landmark, \(G_n(t,\cdot)=\mathsf K_t1=H(t,\cdot)\). If \(s\) is a landmark
and no further landmark lies in \((s,t)\), then
\[
 G_n(t,\cdot)=\mathsf K_{t-s}G_n(s+,\cdot).
\]

At a landmark \(s\in \A_n(b)\), we distinguish the value
\(G_n(s,\cdot)\), which includes paths stopping at \(s\), from the value
immediately afterward. By the definition of \(\tau_b^{(n)}\),
\begin{equation} \label{eq:landmark-update}
 G_n(s+,x)
 =\mathbb P\bigl(\tau_b^{(n)}>s,\ Y_s>x\bigr)
 =\mathbb P\bigl(\tau_b^{(n)}\ge s,\
                   Y_s>x\vee b(s)\bigr)
 =G_n(s,x\vee b(s)).
\end{equation}
When \(b(s)=-\infty\), the update leaves \(G_n(s,\cdot)\) unchanged.
Moreover, \(G_n(t,0)=0\) for all \(t>0\).

Between consecutive landmarks, \(G_n\) satisfies the heat equation.
At a landmark \(s\), its change 
 $G_n(s+,x)-G_n(s,x)$
is nonpositive and vanishes whenever \(x>b(s)\). Since only finitely
many landmarks meet any compact time interval, these observations imply
the distributional identity
\[
 \bigl(\partial_t-\tfrac12\partial_{xx}\bigr)G_n
 =\sum_{s\in \A_n(b)}
   \bigl[G_n(s+,x)-G_n(s,x)\bigr]\,
   \delta_s(dt)\,dx.
\]
Consequently, we have
\begin{equation}\label{eq:landmark-heat}
 \bigl(\partial_t-\tfrac12\partial_{xx}\bigr)G_n\le0
 \quad\text{on }(0,\infty)\times(-\infty,0),
 \qquad
 \bigl(\partial_t-\tfrac12\partial_{xx}\bigr)G_n=0
 \quad\text{in }Q_b.
\end{equation}

We can now prove Proposition \ref{prop:direct}.

\begin{proof}[Proof of Proposition~\ref{prop:direct}]
Write $\tau=\tau_b$ and $q(t)=\Pbb(\tau\ge t)$.
For a fixed $t\in(0,T]$, on $\{\tau=t\}$ continuity of the sample path and upper semicontinuity of $b$ imply $Y_t\le b(t)$.  On the other hand, we have
$Y_s>b(s)$ for $0<s<t$, consequently
\[
 Y_t\ge\limsup_{s\uparrow t}b(s)=b(t).
\]
Hence $\{\tau=t\}\subseteq\{Y_t=b(t)\}$, which is an event of probability zero.  The conclusion is also immediate when $b(t)=-\infty$.  Since
$\Pbb(\tau=0)=0$, $q$ is continuous on $[0,T]$ and $q(0)=1$.

Using the same argument, we can justify that $\{\tau\ge t\}\subseteq\{Y_t\ge b(t)\}$.  As the distribution of $Y_t$ is atomless, we  
know that $G^b(t,\cdot)$ is constant below $b(t)$ and $G^b(t,b(t))=q(t)$. Using the definition, we can directly obtain  $0\le G^b\le1$ and
the spatial monotonicity.  Joint continuity follows from sample-path continuity, the absence of atoms of $\tau$ at positive deterministic times, and the atomless marginals of $Y$.  Furthermore, for $x \le 0$, we have
\begin{equation}\label{eq:direct-initial-uniform}
 0\le H(t,x)-G^b(t,x)
   =\Pbb(\tau<t,Y_t>x)\le1-q(t).
\end{equation}
This proves (D1).  Now by Lemma~\ref{lem:landmark} and bounded convergence we can pass to the limit in \eqref{eq:landmark-heat}, and prove (D3)--(D4).  In particular, interior regularity for the heat equation implies that \(G^b\) is smooth in \(Q_b\).

We provide the comparison result, noticing that the boundary at zero and the initial corner must both be respected.  Let $U,V$ be two viscosity
solutions for the same $b$, and suppose
$U(t_0,x_0)-V(t_0,x_0)=8\delta>0$.  By continuity we may take
$0<t_0<T$ and $x_0<0$.  Choose $t_1\in(0,t_0)$ small enough such that
\[
 \sup_{x\le0}|U(t_1,x)-H(t_1,x)|<\delta,
 \qquad
 \sup_{x\le0}|V(t_1,x)-H(t_1,x)|<\delta.
\]
We then choose \(\eta>0\) sufficiently small such that $\eta<-x_0$, $\sup_{t\in[t_1,t_0]}b(t)<-\eta$ and $\sup_{t\in[t_1,t_0]}U(t,-\eta)<\delta$.
If necessary, we can further reduce \(\eta\) and also require
\[
 |V(t_0,x_0+\eta)-V(t_0,x_0)|<\delta,
 \qquad
 \sup_{x\le-\eta}
 |H(t_1,x+\eta)-H(t_1,x)|<\delta.
\]
Such a choice is possible by the compact separation of \(b\) from
zero, the continuity of \(U\) at \(x=0\), the continuity of \(V\)
at \((t_0,x_0)\), and the uniform continuity of \(H(t_1,\cdot)\).
Finally, choose \(\varepsilon\in(0,1)\) such that
\[
 \varepsilon(t_0-t_1)+\varepsilon^2x_0^2<\delta.
\]
For \(t_1\le t\le t_0\) and \(x\le-\eta\), let
\[
 \Phi(t,x):=U(t,x)-V(t,x+\eta)
       -\varepsilon(t-t_1)-\varepsilon^2x^2.
\]
At $(t_0,x_0)$ this exceeds $6\delta$; at $t=t_1$ it is at most
$3\delta$, at $x=-\eta$ at most $\delta$, and it tends uniformly to
$-\infty$ as $x\to-\infty$.  It therefore attains a positive maximum
at $(t_*,x_*)$ with $t_*>t_1$ and $x_*<-\eta$.

If $x_*\le b(t_*)-\eta$, both arguments lie in the regions where the corresponding functions are constant in \(x\).  Monotonicity of $V$ implies that 
 $\Phi(t_*,b(t_*))-\Phi(t_*,x_*)
 \ge\varepsilon^2\bigl(x_*^2-b(t_*)^2\bigr)>0$, which
contradicts the   maximality.  This case is empty if $b(t_*)=-\infty$.

If \(x_*+\eta>b(t_*)\), by (D3), the function
\((t,x)\mapsto V(t,x+\eta)\) is smooth in a neighborhood of
\((t_*,x_*)\) and satisfies
\[
 \bigl(\partial_t-\tfrac12\partial_{xx}\bigr)
 V(t,x+\eta)=0
\]
at \((t_*,x_*)\). Since \(U\) satisfies (D4) in the viscosity sense and
\(V(t,x+\eta)\) is a smooth solution of the heat equation,
\(\Phi\) satisfies
\[
 \bigl(\partial_t-\tfrac12\partial_{xx}\bigr)\Phi
 \le -\varepsilon+\varepsilon^2<0
\]
in the viscosity sense in a neighborhood of \((t_*,x_*)\).
This contradicts the maximum principle: a function satisfying
\[
\bigl(\partial_t-\tfrac12\partial_{xx}\bigr)\Phi<0
\]
near \((t_*,x_*)\) cannot attain its maximum there when \(t_*>t_1\)
and \(x_*\) is spatially interior. The conclusion also holds when
\(t_*=t_0\), since the maximum principle applies up to the final time.  Thus $U\le V$;
interchanging the two functions proves uniqueness.  The probability
identities in the proposition were established above.
\end{proof}


\subsection{Proof of Theorem~\ref{thm:recovery_law}}


Fix $T>0$ such that $p>0$ on $[0,T]$.  Throughout this subsection $F=V$ and $b=\hat b$.  Recall that \(Q_b=\{(t,x):t>0,\ b(t)<x<0\}\). By Theorem~\ref{thm:wellposedness} and \(F\le p\) give, on \(0<t\le T\), we have
\[
 Q_b=\{(t,x):F(t,x)<p(t),\ x<0\}.
\]
As $F(t,0)=0<p(t)$, upper semicontinuity implies for $0<\epsilon<T$,
\begin{equation}\label{eq:boundary-away-zero}
 \sup_{t\in[\epsilon,T]}b(t)<0.
\end{equation}
The endpoint $b(t)=-\infty$ is allowed in this assertion.

To apply Proposition~\ref{prop:direct}, we also need \(b(t)=b^-(t)\). The upper semicontinuity obtained from the previous section implies only
\(b(t)\ge b^-(t)\). The next lemma rules out a strict inequality by using the positivity of \(-F_x\) in \(Q_b\).

\begin{lemma}\label{lem:left-regular-boundary}
For $0<t\le T$, we have
\[
 b(t)=\limsup_{s\uparrow t}b(s).
\]
\end{lemma}
\begin{proof}
Since \(F<p\) in \(Q_b\), by the variational inequality, we have \(F_t-\frac12F_{xx}=0\) in \(Q_b\) in the viscosity sense.
Using the interior regularity for the heat equation, we know  \(F\in C^\infty(Q_b)\). Hence \(u:=-F_x\) is well defined in \(Q_b\). In addition,
it follows from the spatial monotonicity that \(u\ge0\), and from direct differentiation of the
heat equation that \(u_t-\frac12u_{xx}=0\) in \(Q_b\).

We now justify that $u>0$ throughout $Q_b$. Indeed, suppose $u(t,x)=0$ at an interior point. Choose a finite $a_0$ with $b(t)<a_0<x$.
By upper semicontinuity of $b$, there exists $\delta\in(0,t)$ such that $b(s)<a_0$ for $t-\delta<s\le t$. The strong maximum principle on the rectangle
$(t-\delta,t]\times(a_0,0)$ then implies that $u(s,y)=0$ for all $t-\delta<s<t$ and $a_0<y<0$. Thus $F(s,\cdot)$ would be constant on $(a_0,0)$. Its continuous value at zero is zero, contradicting $F(s,y)\ge p(s)H(s,y)>0$ for $a_0<y<0$.

Upper semicontinuity already implies that $b(t)\ge b^-(t)$, with equality when $b(t)=-\infty$.  Suppose now $b^-(t)<b(t)>-\infty$.  Choose
finite numbers $a<c<d<r<b(t)$ with $b^-(t)<a$.  There is $h>0$ with $t-2h>0$ and $b(s)<a$ for $t-2h\le s<t$.  Therefore \(u\) is strictly positive and satisfies \(u_t-\frac12u_{xx}=0\) in a neighborhood of every point of \([t-h,t)\times[a,r]\).
At time \(t-h\), let $m:=\min_{y\in[a,r]}u(t-h,y)>0$ and $\lambda:=\frac{\pi^2}{2(r-a)^2}$. For $t-h\le s<t, a\le y\le r$, define 
\[
 w(s,y):=m e^{-\lambda(s-t+h)}
          \sin\!\left(\frac{\pi(y-a)}{r-a}\right).
\]
The function \(w\) satisfies \(w_s-\frac12w_{yy}=0\). At the initial time, \(w(t-h,y)\le m\le u(t-h,y)\); on the spatial boundaries, \(w(s,a)=w(s,r)=0\le u(s,a),u(s,r)\). For any \(s_0<t\), the parabolic comparison principle on \([t-h,s_0]\times[a,r]\) therefore implies \(u\ge w\).
As \(s_0<t\) is arbitrary, this holds throughout \([t-h,t)\times[a,r]\).

In particular, for \(y\in[c,d]\) and \(t-h\le s<t\), we have
\[
 u(s,y)\ge c_0:=
 m e^{-\lambda h}
 \min_{z\in[c,d]}
 \sin\!\left(\frac{\pi(z-a)}{r-a}\right)>0.
\]
Consequently, we obtain $
 F(s,c)-F(s,d)=\int_c^d u(s,y)\,dy\ge c_0(d-c)>0$. 
By continuity of \(F\) and letting \(s\uparrow t\), we have
\[
 F(t,c)-F(t,d)\ge c_0(d-c)>0.
\]
But \(c<d<b(t)\), and \(F(t,x)=p(t)\) for every \(x\le b(t)\).
Thus \(F(t,c)=F(t,d)=p(t)\), which is a contradiction.  This proves the lemma.
\end{proof}

A further hypothesis of Proposition~\ref{prop:direct} is that \(\Pbb(\tau_b>0)=1\). The condition that \(b(t)<0\) for $t>0$
does not directly rule out \(\tau_b=0\), since the boundary may approach zero as \(t\downarrow0\). We establish the required
initial survival by comparing \(F\) with the landmark functions.

\begin{lemma}\label{lem:initial-survival}
The constructed boundary satisfies $\Pbb(\tau_b=0)=0$.
\end{lemma}
\begin{proof}
By the dynamic programming inequality, we have that for $0<s<t$, 
\begin{equation}\label{eq:inverse-semigroup-subsolution}
 F(t,\cdot)\le\mathsf K_{t-s}F(s,\cdot).
\end{equation}
Also, it holds that $F(t,x)\le H(t,x)=(\mathsf K_t1)(x)$.
Apply the landmark construction to this boundary.  Before its first
positive landmark, $G_n=H$, so $F\le G_n$.  If this inequality holds
just before a landmark $s$, then for $x\le b(s)>-\infty$,
\[
 G_n(s+,x)=G_n(s,b(s))\ge F(s,b(s))=p(s)=F(s,x).
\]
For $x>b(s)$ the truncation changes nothing; if $b(s)=-\infty$
there is no truncation.  Between landmarks, by
\eqref{eq:inverse-semigroup-subsolution} and positivity of
$\mathsf K$, we have that the inequality is still valid.  Using Induction over the finitely many
landmarks in each compact time interval, we obtain that $F\le G_n$ everywhere.
Finally, by Lemma~\ref{lem:landmark}, we have
\[
 F(t,x)\le G^b(t,x),\qquad
 p(t)=\lim_{x\to-\infty}F(t,x)
       \le\Pbb(\tau_b\ge t)\le1.
\]
Letting $t\downarrow0$, we can draw the conclusion.  
\end{proof}

We now check that the variational-inequality solution \(F\) satisfies (D1)--(D4) for
this same boundary.

\begin{lemma}\label{lem:inverse-is-direct}
The function $F$ is a viscosity solution associated with $b$
on $[0,T]$.
\end{lemma}
\begin{proof}
By the previously proved bounds, we have
\[
 0\le H(t,x)-F(t,x)\le(1-p(t))H(t,x)\le1-p(t),
\]
which proves the uniform initial condition.  The remaining conditions in (D1)--(D2) follow from
Theorem~\ref{thm:wellposedness} and the definition of \(b\). When \(b(t)=-\infty\), the stopping region at time \(t\) is empty,
so there is no constancy condition to check in (D2). In \(Q_b\), we have \(F<p\), so the variational inequality implies
\(F_t-\frac12F_{xx}=0\) in the viscosity sense, hence proves (D3).
Its subsolution inequality implies \(F_t-\frac12F_{xx}\le0\) on \((0,T)\times(-\infty,0)\), which proves (D4).
\end{proof}

\begin{proof}[Proof of Theorem~\ref{thm:recovery_law}]
For $T>0$ with $p>0$ on $[0,T]$, Lemmas
\ref{lem:left-regular-boundary} and \ref{lem:initial-survival}, together
with \eqref{eq:boundary-away-zero}, verify every boundary hypothesis of
Proposition~\ref{prop:direct}.  Lemma~\ref{lem:inverse-is-direct} and
uniqueness therefore yield
\[
 F(t,x)=G^b(t,x)=\Pbb(\tau_b\ge t,Y_t>x),\qquad0<t\le T.
\]
Letting $x\to-\infty$, we have that  $\Pbb(\tau_b\ge t)=p(t)$. Using the
absence of atoms, we get $\Pbb(\tau_b>t)=p(t)$.  The same equality at
$t=0$ comes from Lemma~\ref{lem:initial-survival}.
If $T_*:=\inf\{t:p(t)=0\}<\infty$, apply the preceding argument
for every $T<T_*$.  By the continuity of $p$, we have
\[
 \Pbb(\tau_b\ge T_*)=\lim_{t\uparrow T_*}\Pbb(\tau_b\ge t)
 =\lim_{t\uparrow T_*}p(t)=0.
\]
Both $p$ and $F$ vanish from $T_*$ onward, so all claimed identities
extend to these times as well.
\end{proof}

The preceding theorem provides the candidate stopping time directly:
\begin{equation}\label{eq:candidate-tau}
 \hat\tau=\tau_{\hat b}=\inf\{t>0:Y_t\le \hat b(t)\}\sim\mu.
\end{equation}

\section{Proof of the optimality} \label{sec:proof3}

\subsection{Discrete case}

Fix $0=t_0<t_1<\cdots<t_N=H$, choose thresholds
$b_1,\ldots,b_{N-1}\in[-\infty,0]$, set $b_N=0$, and define the stopping time.
\[
 \tau^\pi:=\min\{t_k:1\le k\le N,\ Y_{t_k}\le b_k\}.
\]
The condition $b_N=0$ forces every path to stop by $H$. Let us denote by $\mathcal{T}^\pi$ the set of stopping times that take values in $\{t_0,\dotso,t_N\}$. The following lemma proves that $\tau^\pi$ is a solution of an optimal stopping problem. The argument is inspired by the construction of the Snell envelop. 

\begin{lemma}\label{lem:grid}
The stopping time $\tau^\pi$ maximizes $\E[M_\sigma]$ over all stopping times $\sigma$ such that 
$ \sigma \in \mathcal{T}^\pi$ and $ \sigma  \sim \tau^\pi$.
\end{lemma}

\begin{proof}
We claim that there exists a sequence of non-negative numbers $\psi_{t_0}, \dotso,\psi_{t_{N-1}}, \psi_{t_N}=0$, so that the stopping time $\tau^{\pi}$ solves the unconstrained optimal stopping problem 
\begin{align}\label{eq:unconstrained_OS}
    \max_{\sigma \in \mathcal{T}^{\pi}} \E [M_{\sigma}+\psi_{\sigma}].
\end{align}
More precisely, the Snell envelope of  $\bar M_{t_k}:= M_{t_k} + \psi_{t_k}$, defined by
\begin{align*}
    S_{t_k}:= \esssup_{\sigma \in \mathcal{T}^\pi, \, \sigma \geq t_k} \mathbb E[\bar M_{\sigma} \mid \mathcal{F}_{t_k}], \quad k=0,\dotso, N,
\end{align*}
is a super-martingale.  The hitting time 
\begin{align*}
\tau:=\inf\{k \in \{0,\dotso,N\} : \, S_{t_k}=\bar M_{t_k} \}
\end{align*}
solves the optimal stopping problem \eqref{eq:unconstrained_OS}, 
and the stopped process $\bar M^{\tau}_{t_k}$ is a martingale. We will construct a sequence of $\{\psi_{t_k}\}_{k \in 0,\dotso, N\}}$ such that $\tau^\pi=\tau$. Therefore according to the optional stopping theorem, for any $\sigma \in \mathcal{T}^\pi$ with $\sigma \sim \tau^\pi$
\begin{align*}
 \E[ M_{\tau^\pi}+\psi_{\tau^\pi}]=\E [M_{t_0}+\psi_{t_0}] \geq \E [M_{\sigma}+\psi_{\sigma}].
\end{align*}
Since $\tau^\pi \sim \sigma$, it yields $\E[M_{\tau^\pi} ]\geq \E[M_{\sigma}]$, which proves the lemma. In the rest, we prove the existence of such $\psi_{t_0},\dotso, \psi_{t_N}=0$ that $\tau^\pi$ solves \eqref{eq:unconstrained_OS}.

For an arbitrary sequence of non negative numbers $\psi_{t_0}, \dotso, \psi_{t_{N-1}}, \psi_{t_N}=0$, the Snell envelope of \eqref{eq:unconstrained_OS} can be constructed iteratively 
\begin{align*}
    S_{t_N}&=\bar M_{t_N}=M_{t_N}, \\
    S_{t_k}&=\max \{\bar M_{t_k}, \E[S_{t_{k+1}} \mid \mathcal{F}_{t_k}] \}=\max \{M_{t_k}+\psi_{t_k}, \E[S_{t_{k+1}} \mid \mathcal{F}_{t_k}] \}, \quad k=N-1,\dotso, 0. 
\end{align*}
Let $(W_t)$ be a Brownian motion with $W_0=0$ that is independent of $(B_t)$. For any $k \in \{0,\dotso,N-1\}$ and $y \leq 0$, let us define 
\begin{align}\label{eq:conditional_maximum_increment}
    L_k(y):=\mathbb E\left[ \left(y+ W^*_{t_{k+1}-t_{k}}\right)^+ \right],
\end{align}
which is clearly a positive strictly increasing continuous function. By the Markovian property of $(Y_t,M_t)$, it can be easily verified that 
$\E[M_{t_{k+1}} \mid \mathcal{F}_{t_k}]= M_{t_k}+L_k(Y_{t_k})$.

Let us prove by backward induction that for $k=N-1,\dotso, 0$, there exists a strictly increasing continuous function $w_k:(-\infty, 0] \to (0,\infty)$ such that 
\begin{align}\label{eq:induction}
\mathbb E[S_{t_{k+1}} | \, \mathcal{F}_{t_k}]=M_{t_k}+w_k(Y_{t_k}).
\end{align}
For $k=N-1$, we simply take $w_{N-1}=L_{N-1}$, which satisfies the property due to the explicit formula \eqref{eq:conditional_maximum_increment}. Suppose \eqref{eq:induction} holds for $k$, and let us prove it for $k-1$. According to the induction hypothesis and the construction of the Snell envelope, 
\begin{align*}
    \mathbb E\left[S_{t_k} | \, \mathcal{F}_{t_{k-1}} \right]& = \E \left[\max\{M_{t_k}+\psi_{t_k}, M_{t_k}+w_k(Y_{t_k})\}    \,  | \, \mathcal{F}_{t_{k-1}} \right] \\
    & = \E[M_{t_k}  \,| \, \mathcal{F}_{t_{k-1}}]+ \E[\max \{\psi_{t_k},w_k(Y_{t_k}) \} | \, \mathcal{F}_{t_{k-1}}] \\
    &= M_{t_{k-1}}+L_{k-1}(Y_{t_{k-1}})+ \E[\max \{\psi_{t_k},w_k(Y_{t_k}) \} | \, \mathcal{F}_{t_{k-1}}].
\end{align*}
Note that $Y_{t_k}$ and $\mathcal{F}_{t_{k-1}}$ are conditionally independent given $Y_{t_{k-1}}$. Therefore we have  
\begin{align*}
\hat L_{k-1}(y):=&\E\left[\max \left\{\psi_{t_k},w_k\left(y+W_{t_k-t_{k-1}} -(y+W^*_{t_k-t_{k-1}}\right)^+  )  \right\} \right], \, \, \,  y\leq 0, \\
  \hat L_{k-1}(Y_{t_{k-1}})=&  \E[\max \{\psi_{t_k},w_k(Y_{t_k}) \} | \, \mathcal{F}_{t_{k-1}}].
\end{align*}
From the explicit formula, $\hat L_{k-1}: \mathbb R_{-} \to [0,\infty]$ is continuous non-decreasing function. Hence  $w_{k-1}(y):=L_{k-1}(y)+\hat L_{k-1}(y)$ is positive, strictly increasing, and continuous. So that $\mathbb E\left[S_{t_k} | \, \mathcal{F}_{t_{k-1}} \right]=M_{t_{k-1}}+w_{k-1}(Y_{t_{k-1}})$ completes the induction. 

By the classical theory of optimal stopping, the optimal stopping 
\begin{align*}
    \hat \tau:= \inf \{t_k : \,M_{t_k}+\psi_{t_k} \geq S_{t_k}\}
\end{align*}
solves \eqref{eq:unconstrained_OS}. It remains to choose $\psi_{t_1}, \dotso, \psi_{t_{N-1}}, \psi_{t_N}$ so that $\hat \tau=\tau^\pi$. Note that $M_{t_k} + \psi_{t_k} \geq S_{t_k}$ is equivalent to $\psi_{t_k} \geq w_k(Y_{t_k})$. As the construction of $w_{k-1}$ depends on $\psi_{t_k}$ and $w_k$, we choose $\psi_{t_{N-1}}=w_{N_1}(b_{N-1}), \dotso, \psi_{t_k}=w_k(b_k)$ in reverse order. Then as $w_k$ is continuous strictly increasing, $\psi_{t_k} \geq w_k(Y_{t_k})$ is equivalent to $Y_{t_k} \leq b_k$, which completes the proof. 
\end{proof}

\begin{remark}[Relation to martingale and geometric methods]
The preceding proof establishes optimality using a discrete-time Snell envelope, in the spirit of the martingale arguments for optimal Skorokhod embedding in \cite{COT} and the dual approach of \cite{GHLT}. The key distinction lies in the constraint: prescribing the distribution of the stopped position, $B_\sigma \sim \mu$, fixes the expectation of the Lagrange multiplier term $\lambda(B_\sigma)$, whereas prescribing the stopping-time distribution, $\sigma \sim \mu$, fixes the expectation of $\psi(\sigma)$. The latter observation allows us to transfer optimality for the modified reward $M_\sigma+\psi(\sigma)$ to the original constrained problem.

In the next subsection, we extend optimality to continuous time using convergence of the stopping times and uniform integrability,
without requiring convergence of the functions $\psi$.  This also differs from the geometric approach of \cite{BEES2018},
which uses a monotonicity principle and the exclusion of stop-go pairs to identify the structure of optimizers. Here the candidate
boundary is constructed separately and its optimality is then verified directly.
\end{remark}

\subsection{Passage to the rough barrier}

Fix $T>0$, put $\theta:=\hat\tau\wedge T$, and let
\[
 g(t):=\Pbb(\theta>t),\qquad
 H:=\sup\{t:g(t)>0\}\le T.
\]
Atomlessness of $\mu$ implies $\theta>0$ almost surely, so $H>0$.  Define the
capped boundary
\[
 b^T(t):=
 \begin{cases}
  \hat b(t),&0<t<H,\\
  0,&t=0\ \text{or}\ t\ge H.
 \end{cases}
\]
The boundary $b^T$ is upper semicontinuous in $t$, and its hitting time equals to $\theta$ almost surely.
Indeed, if $H=T$, the condition $b^T(T)=0$ stops every path that has survived
until $T$.  If $H<T$, right-continuity of $g$ and the atomlessness of $\mu$ imply $\hat \tau < H$  $a.s.$, and hence $\theta=\hat \tau$ equals to the hitting time of $b^T$. 

For each $n\ge1$, let $\Delta_n:=H2^{-n}$ and define the dyadic grid
$\{t_k^n=k\Delta_n:0\le k\le2^n\}$. Choose  the deterministic thresholds
$b_k^n\in[-\infty,0)$ recursively, so that
\[
\Pbb(Y_{t^n_1}>b_1^n)=g(t_1^n), \quad  \Pbb\left(Y_{t_k^n}>b_k^n\mid
 Y_{t_j^n}>b_j^n,\ 1\le j<k\right)
 =\frac{g(t_k^n)}{g(t_{k-1}^n)}, \, \,  \, k=1,\dotso, 2^n-1,
\]
and put $b_{2^n}^n=0$.  The existence of such $b_k^n$ is implied by the fact that 
\[ [-\infty,0) \ni y \mapsto  \Pbb\left(Y_{t_k^n}>y\mid
 Y_{t_j^n}>b_j^n,\ 1\le j<k\right) \in (0,1]\]
is continuous and decreasing. 
Define the deterministic boundary $b_n:[0,\infty]\to[-\infty,0]$ by
\[
 b_n(t):=\begin{cases}
 b_k^n,&t=t_k^n,\quad 1\le k<2^n,\\
 -\infty,&0<t<H,\quad t\notin\pi_n,\\
 0,&t=0\ \text{or}\ H\le t\le\infty.
 \end{cases}
\]
This boundary $b_n(t)$ is upper semicontinuous, and its hitting time is
\begin{align}\label{eq:tau_n}
 \tau_n:=\inf\{t>0:Y_t\le b_n(t)\}
       =\min\{t_k^n:1\le k\le2^n,\ Y_{t_k^n}\le b_k^n\}.
\end{align}

Since $0<\theta\le H$ almost surely, define
\[
 \theta_n:=\min\{t_k^n:1\le k\le2^n,\ \theta\le t_k^n\}
          =\Delta_n\left\lceil\frac{\theta}{\Delta_n}\right\rceil
          \quad\text{almost surely}.
\]
By Induction in $k$, we have
\[
 \Pbb(\tau_n>t_k^n)=g(t_k^n)=\Pbb(\theta_n>t_k^n),
 \qquad 0\le k<2^n,
\]
and hence
$\tau_n\stackrel d=\theta_n$.  As
$\theta_n \to \theta$ almost surely, $\theta_n$ and $\tau_n$ converge weakly to $\theta$.

Now the key to proving Theorem~\ref{thm:main} is the following stability property of inverse first-passage time problem. 
\begin{lemma}
    With $\tau_n$ defined in \eqref{eq:tau_n} and $\theta=\hat \tau \wedge T$, it holds that $\tau_n\to \theta$ in probability.
\end{lemma}
This lemma can be deduced from \cite{KlumpKolb2023}. Indeed, using the fact that $\tau_n$ converge weakly to $\theta$, and applying \cite[Lemma 4.4]{KlumpKolb2023} to the corresponding inverse first-passage boundaries yields the Hausdorff convergence of $b_n$ to $b^T$. Since both $b_n$ and $b^T$ are set equal to $0$ on $[H, \infty)$, this extends to the full boundaries. Hence,  \cite[Lemma 3.13(ii)]{KlumpKolb2023} then implies
\[
    \tau_n \to \theta
    \qquad \text{in probability}.
\]

\begin{proof}[Proof of Theorem~\ref{thm:main}]
Let $\rho$ be any stopping time with distribution $\mu$. We prove that $\E[M_{\rho}] \leq \E[M_{\hat \tau}]$. 

Fix $T>0$ and set $\sigma:=\rho\wedge T$.  Then
$\sigma\stackrel d=\theta$ and $\sigma\le H$ almost surely. For each $n$, defining the stopping time
\[  \sigma_n:=\min\{t_k^n:1\le k\le2^n,\ \sigma\le t_k^n\},
          \]
we have that 
$\sigma_n\stackrel d=\theta_n\stackrel d=\tau_n$.  Lemma~\ref{lem:grid}
therefore implies
\[
 \E[M_{\sigma_n}]\le\E[M_{\tau_n}].
\]
As $\sigma_n \to \sigma$ almost surely, continuity
of $M$ and domination by $M_T\in L^1$ imply
$M_{\sigma_n}\to M_\sigma$ in $L^1$.  Likewise,
$\tau_n\to\theta$ in probability and uniform continuity of every path of
$M$ on $[0,T]$ imply that $M_{\tau_n}\to M_\theta$ in probability.  Since
$0\le M_{\tau_n}\le M_T$ with $M_T \in L^1$, the family $(M_{\tau_n})_n$ is uniform integrability. Hence, by Vitali's convergence
theorem, $M_{\tau_n} \to M_{\theta}$ in $L^1$.  Passing to the limit yields
\[
 \E[M_{\rho\wedge T}]\le\E[M_{\hat\tau\wedge T}].
\]
Finally let $T\to \infty$.  Since
$M_{\rho\wedge T}\to M_\rho$ and
$M_{\hat\tau\wedge T}\to M_{\hat\tau}$, monotone convergence
proves the result, with both expectations interpreted in $[0,\infty]$.  
\end{proof}

\end{document}